\documentclass[12pt,reqno]{amsart}

\usepackage{amsaddr}

\usepackage{amsthm}
\usepackage{amsmath}
\usepackage{times}
\usepackage{graphicx}
\usepackage{color}
\usepackage{multirow}
\usepackage[authoryear]{natbib}
\usepackage{rotating}
\usepackage{bbm}
\usepackage{latexsym}
\usepackage{fullpage}

\usepackage{microtype}
\usepackage{graphicx}
\usepackage{subfigure}
\usepackage{booktabs} 

\usepackage{hyperref}
\usepackage{mathtools}
\DeclarePairedDelimiter{\abs}{\lvert}{\rvert}

\usepackage{amsmath,amssymb,amsthm}
\usepackage{bm}

\usepackage{comment}
\usepackage{algorithm}
\usepackage{algorithmic}
\usepackage{booktabs}
\usepackage{wrapfig,url}
\usepackage{multirow}
\usepackage{lipsum}
\usepackage{caption}
\usepackage{booktabs}

\usepackage{enumitem}

\newcommand{\N}{\mathbb{N}}

\newcommand{\bG}{\mathbb{G}}

\newcommand{\bE}{\mathbb{E}}

\renewcommand{\le}{\leqslant}

\newcommand{\Ep}{\mathbb{E}}
\renewcommand{\Pr}{\mathrm{Pr}}

\newcommand{\mF}{\mathcal{F}}

\newcommand{\mA}{\mathcal{A}}
\newcommand{\mB}{\mathcal{B}}
    
\newcommand{\mD}{\mathcal{D}}

\newcommand{\mN}{\mathcal{N}}
\newcommand{\mJ}{\mathcal{J}}

\newcommand{\mX}{\mathcal{X}}

\newcommand{\lip}{\operatorname{Lip}}

\DeclareMathOperator{\Exp}{Exp}

\newcommand{\cvd}{\stackrel{d}{\rightarrow}}

\renewcommand{\hat}{\widehat}
\renewcommand{\tilde}{\widetilde}

\renewcommand{\footnotesize}{\fontsize{12pt}{12pt}\selectfont}

\allowdisplaybreaks

\newtheorem{theorem}{Theorem}
\newtheorem{lemma}{Lemma}

\newtheorem{corollary}{Corollary}
\newtheorem{assumption}{Assumption}

\newtheorem{remark}{Remark}

\usepackage{color}
\newcommand{\rc}{\color{red}}

\title{Hypothesis Test and Confidence Analysis with \\Wasserstein Distance on General Dimension}

\author{Masaaki Imaizumi$^{\displaystyle 1, \displaystyle 3}$, Hirofumi Ota$^{\displaystyle 2}$, Takuo Hamaguchi$^{\displaystyle 1}$}

\address{$^1$The University of Tokyo, $^2$Rutgers University, $^3$RIKEN AIP}
\thanks{\textit{Contact:} \texttt{imaizumi@g.ecc.u-tokyo.ac.jp}}

\begin{document}

\maketitle

\begin{abstract}
We develop a general framework for statistical inference with the {$1$-Wasserstein distance}. Recently, the Wasserstein distance has attracted considerable attention and has been widely applied to various machine learning tasks because of its excellent properties. However, hypothesis tests and a confidence analysis for the Wasserstein distance have not been established in a general multivariate setting. This is because the limit distribution of the empirical distribution with the Wasserstein distance is unavailable without strong restriction. To address this problem, in this study, we develop a novel \textit{non-asymptotic Gaussian approximation} for the empirical {$1$-Wasserstein distance}. Using the approximation method, we develop a hypothesis test and confidence analysis for the empirical {$1$-Wasserstein distance}. Additionally, we provide a theoretical guarantee and an efficient algorithm for the proposed approximation. Our experiments validate its performance numerically.
\end{abstract}

\section{Introduction}

    \textit{The Wasserstein distance} \citep{vaserstein1969markov} has attracted significant attention as a criterion for the discrepancy between two probability measures.
    It is based on the natural mathematical formulation of the optimal transportation \citep{villani2008optimal}.
    Recently, it has been demonstrated that the transportation notion is suitable for pattern analysis and feature extraction from data.
    The Wasserstein distance can be used to measure the discrepancy in the data distributions with singular supports, such as real image data, whereas standard measures, such as the Kullback-Leibler divergence, fail at doing so. 
    Owing to this advantage, the Wasserstein distance has been utilized extensively in machine learning and related fields, such as in generative models \citep{arjovsky2017wasserstein}, supervised learning \citep{frogner2015learning}, medical image analysis \citep{ruttenberg2013quantifying}, genomics \citep{evans2012phylogenetic}, and computer vision \citep{ni2009local}.

    Despite this significance, \textit{statistical inference} (e.g., a goodness-of-fit test) using the Wasserstein distance is severely hindered.
    A goodness-of-fit test is a fundamental and vital method to evaluate the uncertainty of the observed distributions rigorously, and it is extensively studied considering various distances such as the Kullback-Leibler divergence \citep{massey1951kolmogorov,vasicek1976test,song2002goodness,lloyd2015statistical}.
    However, statistical inference using the Wasserstein distance is possible only with univariate data \citep{munk1998nonparametric,del2000contributions,ramdas2017wasserstein,bernton2017inference,del2019central} or discrete-valued data \citep{sommerfeld2018inference,tameling2017empirical,bigot2017central}.
    For general multivariate data, a Cases where the distribution is Gaussian have been investigated by  \citet{del2019central2}, and 
    {a one-sample test is developed by  \citet{hallin2021multivariate}.}
    However, it is difficult to conduct statistical inferences in general situations including non-Gaussian cases or a two-sample test.
    This limitation regarding statistical inference can be attributed to the obscure limit distribution of the Wasserstein distance, and addressing this difficulty has been an important open question as described in Section 3 of a review paper \citep{panaretos2018statistical}.

    In this study, we aim to overcome this limitation by developing a \textit{non-asymptotic Gaussian approximation} \citep{chernozhukov2013gaussian,chernozhukov2014gaussian} for the empirical {$1$-Wasserstein distance}.
    We approximate the distance using a supremum of Gaussian processes and prove that it is a consistent approximator.
    Importantly, the approximation does not require a limit distribution, hence we can avoid the problem of an unavailable limit.
    Consequently, we can evaluate the uncertainty of the empirical {$1$-Wasserstein distance} for general multivariate data without substantial restrictions.
    Intuitively, we show that it is possible to infer the empirical {$1$-Wasserstein distance} with a non-asymptotic distributions even when it does not have asymptotic distributions due to its supremum over the large function space.
    For practical applications, we propose an efficient multiplier bootstrap algorithm to calculate the approximator.
    Based on the approximation scheme, we develop a goodness-of-fit test that can control the type I error, and also demonstrate its performance through several experiments.

    The contributions of this paper can be summarized as follows:
    \begin{enumerate}
        \setlength{\parskip}{0cm}
        \setlength{\itemsep}{0cm}
        \item[(i)] An approximation scheme is developed for the empirical {$1$-Wasserstein distance}.
        This scheme is applicable to general multivariate data without strong assumptions, and answers the question of statistical inference using the {$1$-Wasserstein distance}.
        \item[(ii)] A multiplier bootstrap method is proposed to conduct the statistical inference, which is computationally more efficient than other bootstrap and numerical methods.
        \item[(iii)] One-sample and two-sample hypothesis test schemes and a model selection algorithm are developed based on the {$1$-Wasserstein distance}.
    \end{enumerate}

\subsection{Notation}
The $j$-th element of vector $b \in \mathbb{R}^d$ is denoted by $b_{j}$, and $\|\cdot\|_q := (\sum_j b_j^q)^{1/q}$ is the $q$-norm ($q \in [0,\infty]$).
For a matrix $A\in \mathbb{R}^{d \times d'}$, $\|A\|_q := \|\mbox{vec}(A)\|_q$ where $\mbox{vec}(\cdot)$ is a vectorization operator.
$\delta_x$ is the Dirac measure on $x$.
For $x,x' \in \mathbb{R}$, $x \wedge x' := \min\{x,x'\}$.
All proofs of theorems and lemmas are presented in the supplementary material.

\section{Preparation}

First, we formally define the Wasserstein distance as the distance between two probability measures obtained using transportation between the measures.
Let $(\mX, d)$ be a complete metric space with metric $d:\mX \times \mX \to \mathbb{R}_+$. 
For $p \geq 1$, the $p$-Wasserstein distance between two Borel probability measure $\mu_1$ and $\mu_2$ is defined as {
\begin{align*}
    W_p(\mu_1, \mu_2) :=  \inf_{ \nu \in \Pi(\mu_1, \mu_2)} \left(\int_{\mX\times \mX}d(x, y)^p\nu(dx, dy) \right)^{1/p}
\end{align*}
}
where $\Pi(\mu_1, \mu_2)$ is the set of all Borel probability measures on $\mX \times \mX$ with marginals $\mu_1$ and $\mu_2$. 

The formal problem statement of this study is as follows.

\textbf{Problem Formulation}:
We aim to approximate the empirical {$1$-Wasserstein distance $W_1(\mu_n,\mu)$.}
Let $\mu$ be a probability measure on a sample space $\mX=[0,1]^d$ and $d(x,x') = \|x-x'\|_2$, and $\mD_n^X := \{X_1,...,X_n\}$ be a set of $n$ independent and identical observations from $\mu$.
Let $\mu_n := n^{-1}\sum_{i=1}^n \delta_{X_i}$ be an empirical probability measure.
Regardless to say, {$W_1(\mu_n,\mu)$} is a random variable owing to the randomness of $\mD_n^X$.
We are interested in approximating the distribution of {$W_1(\mu_n,\mu)$} with a tractable random variable.
{
Rigorously, we aim to derive a random variable $Z_n$ and a scale term $R_n \in \mathbb{R}$ depending on $n$ such that 
\begin{align*}
    \abs*{Z_n -  R_n W_1(\mu_n,\mu)} = o_P(1),
\end{align*}
as $n \to \infty$.
In Section \ref{sec:non-asymptotic-approx} for introducing our methodology, we will give a specific form of $Z_n$ as a maximum value of a Gaussian process, and $R_n = \sqrt{n/S}$ where $S$ is a number of parameters of models for approximating {$W_1(\mu_n,\mu)$}.
}
Such the approximation for {$W_1(\mu_n,\mu)$} is necessary for various statistical inference methods, such as a goodness-of-fit test and a confidence analysis.

\subsection{Preparation}

\textbf{Dual form of 1-Wasserstein distance:}
The {$1$-Wasserstein distance} has the following duality \citep{villani2008optimal}:
\begin{align}
    {W_1(\mu_1, \mu_2)} = \sup _ { \phi \in \lip(\mX) } \int \phi(y)\mu_2 ( dy ) \ - \int \phi ( x )  \mu_1 ( dx ), \label{eq:dual}
\end{align}
where $\lip(\mX)$ is a set of Lipschitz-continuous functions on $\mX$ with Lipschitz constants as $1$. 

\textbf{Functions by deep neural network:}
We provide a class of functions in the form of deep neural networks (DNNs).
Let $\Xi(L, S)$ be a class of functions by DNN with $L$ hidden layers, and $S$ non-zero parameters (edges).
Let $A_\ell$ and $b_\ell$ be a matrix and a vector parameter, respectivelly, for the $\ell$-th layer, and $\sigma_b(x):=\sigma(x+b)$ be a ReLU activation function with a shift parameter $b$.
Then, $\Xi(L,S)$ is a set of functions $f:[0,1]^d \to [0,1]$ with the following form
\begin{align*}
    f(x) = A_{L+1} \sigma_{b_{L}}A_{L} \sigma_{b_{L-1}} \cdots A_{2} \sigma_{b_{1}}A_{1}x + b_{L+1},
\end{align*}
such that $  \sum_{\ell =1}^{L+1} (\|A_{\ell}\|_0 + \|b_\ell\|_0)=S $.%

\subsection{Related Work and Difficulty of General Case} \label{sec:related}

Numerous studies have been conducted on statistical inference using the Wasserstein distance in restricted settings.
However, determining statistical inference for general multivariate data has remained unsolved, and our study aims to addresses the question.
We briefly review the related studies and describe the problem source.

\textbf{Univariate case}: 
When $X_i$'s are univariate (i.e. $\mX \subseteq \mathbb{R}$), the Wasserstein distance is written as an inverse of the distribution function. 
Let $F$ be a distribution function of data and $F_n$ be an empirical distribution of generated data.
Then, we can derive {$W_p(F,F_n) = (\int_0^1|F^{-1}(s)-F_n^{-1}(s)|^p ds)^{1/p}$} (see \citep{ambrosio2008gradient}).
Several studies \citep{munk1998nonparametric,del1999tests,ramdas2017wasserstein,del2019central} have derived an asymptotic distribution of $W(F_n,F)$ as a limit of the process with $F_n^{-1}$ and $F^{-1}$.
\cite{bernton2017inference} developed an inference method based on an extended version of the limit distribution.
Obviously, the inverses $F_n^{-1}$ and $F^{-1}$ are valid only in the univariate case, and hence, are not applicable to multivariate cases.

\textbf{Discrete case}:
Conisdering that $X_i$ are discrete random variables (i.e. $\mX$ is a finite space),  \cite{sommerfeld2018inference} showed that {$\sqrt{n}W_p(\mu_n,\mu) \cvd \max_{u \in U} \langle Z,u \rangle $}, where $Z$ is a centered Gaussian random variable and $U$ is a suitable convex set.
Some works \citep{tameling2017empirical,bigot2017central} have inherited this result and developed inference methods.

\begin{figure}
  \begin{center}
   \includegraphics[width=0.45\textwidth]{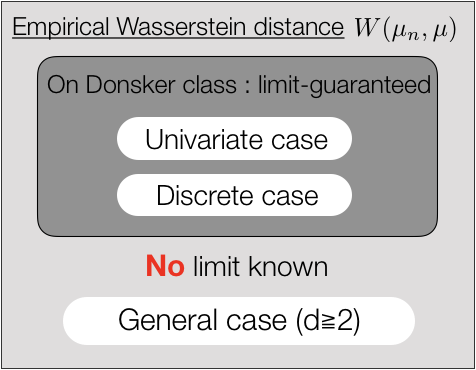}
  \end{center}
  \caption{Class of empirical Wasserstein distances. We can utilize a limit distribution when an index set for processes lies in a Donsker class when $X_i$ are discrete or univariate. However, if $X_i$ are multivariate, { $W_1(\mu_n,\mu)$} is not a process on a Donsker class, and hence, we cannot use its limit distribution. \label{fig:class}}
\end{figure}


\textbf{Multi-dimensional case}:
For the $\mX \subset \mathbb{R}^d$ with $d \geq 2$, a limit distribution of {$W_p(\mu_n,\mu)$} is not available, hence it is difficult to evaluate its uncertainty, except a Gaussian case: \citep{del2019central2} showed that $W_2(\mu_n,\mu)$ has asymptotic normality when $\mu$ is Gaussian.
Based on the dual form \eqref{eq:dual}, {the empirical $1$-Wasserstein distance $W_1(\mu_n,\mu)$} is regarded as an empirical process on $\lip(\mX)$.
However, based on the empirical process theory \citep{vdVW1996}, $\lip(\mX)$ with $d \geq 2$ is considered to be significantly broad, and thus the empirical process does not converge to a known limit distribution.
In other words, {$W_1(\mu_n,\mu)$} is not guaranteed to have a tractable limit distribution as $n \to \infty$.
It is based on the theory of the \textit{Donsker class}, and the limit of {$W_1(\mu_n,\mu)$} is intractable as $\lip(\mX)$ does not belong to the Donsker class with $d \geq 2$ (Summarized in Chapter 2 of \cite{vdVW1996}).
The fact is illustrated in Figure \ref{fig:class}.
We note that the bootstrap approaches (e.g., an empirical bootstrap) are not validated due to the limit problem. 

{
\citet{hallin2021multivariate} developed a goodness-of-fit test with the Wasserstein distance in a different way. 
The method simulates a fixed null distribution, which allows for one-sample tests under general multivariate distributions. 
Since this approach is not based on an asymptotic distribution of the empirical Wasserstein distance, the differences appear in the availability of confidence intervals and guarantees for the bootstrap method.
}

\textbf{Concentration inequality-type interval}: 
A concentration inequality-based method is an alternative approach for confidence intervals.
Several studies \citep{fournier2015rate,weed2019sharp} showed that {$W_p(\mu_n,\mu)$} on $\mathbb{R}^d$ is upper bounded by $C n^{-1/d}$ with high-probability and some constant $C$, hence it is possible to construct a confidence interval with its width $O(n^{-1/d})$.
However, this approach needs to estimate unknown constants. 

\textbf{Extension of Wasserstein distances}
In recent years, under several extended versions of Wasserstein distances, their asymptotic distribution has been derived in general cases.
For example, the sliced Wasserstein distance \citep{nadjahi2019asymptotic}, the smoothed Wasserstein distance \citep{goldfeld2020asymptotic,chen2021asymptotics}, the entropic regularized Wasserstein distance \citep{mena2019statistical,bigot2019central}, and the projection robust Wasserstein distance \citep{lin2021projection}.
Although these distances are different from the original Wasserstein distance, we describe them here for reference.

\section{Non-Asymptotic Approximation}   \label{sec:non-asymptotic-approx}

First, we derive a random variable that approximates {$W_1(\mu_n,\mu)$} (Theorem \ref{thm:GAR_kolmogorov_2} and Corollary \ref{cor:main}).
Next, we present a computationally efficient algorithm (Algorithm \ref{alg:GMB}) and its theoretical validity (Theorem \ref{thm:GMB}).

\subsection{Approximation Theory} \label{sec:approx}
Our approximation scheme contains three steps: (i) approximation of $\lip(\mX)$ in \eqref{eq:dual}, (ii) development of a non-asymptotic Gaussian approximation for {$W_1(\mu_n,\mu)$}, and (iii) combine of the first two steps.

\textbf{Step (i). Approximation by DNNs}:
To calculate the supremum on $\lip(\mX)$ on \eqref{eq:dual}, we introduce $\Xi(L,S)$ and represent the Lipschitz functions by DNNs.
We define a restricted functional class by DNNs with a parameter $\gamma > 0$ as $\Xi^{\lip,\gamma}(L,S) \subset \Xi(L,S)$ such that $|f(x)-f(x')|\leq \|x-x'\|_2 + \gamma,\forall x,x' \in \mX$ holds for any $f \in \Xi^{\lip,\gamma}(L,S)$.
Then, we define {an approximated $1$-Wasserstein distance as}
\begin{align}
    {\hat{W}_1(\mu_1,\mu_2)} := \sup_{f \in \Xi^{\lip,\gamma}(L,S)} \Ep_{\mu_2}[f(X)] -\Ep_{\mu_1}[f(X)], \label{def:wasserstein_dnn}
\end{align}
for given $L$ and $S$.
In the following theorem, we show that {$\hat{W}_1(\mu_1,\mu_2)$} is arbitrarily close to {$ W_1(\mu_1,\mu_2)$} for any $\mu_1$ and $\mu_2$ as $S$ increases with finite $L$.
\begin{lemma}\label{lem:NN-approx}
    For any probability measures $\mu_1$ and $\mu_2$ on $\mX$ and for all sufficiently large $S \in \N, L = O((1+S)(1+1/d))$ and $\gamma = O(S^{-1/d})$, there exists a constant $c>0$ such that
    {
    \begin{align*}
            \abs*{ \hat{W}_1(\mu_1,\mu_2)- W_1(\mu_1,\mu_2)}  \le c S^{-1/d}. 
    \end{align*}
    }
\end{lemma}
We utilize DNNs for approximating $\lip(\mX)$ owing to the following reasons.
The approximation by DNNs is often employed in the field of generative models, particularly for the Wasserstein-GAN \citep{arjovsky2017wasserstein}. Hence several convenient computational algorithms are developed, such as the spectral normalization \citep{miyato2018spectral}.
{
Additionally, DNNs theoretically and computationally work well with high-dimensional $\mX$, unlikely basis function approaches such as the Fourier and wavelet bases.
In detail, DNNs have a better error rate than a linear sum of basis functions for approximating several class of functions \citep{schmidt2017nonparametric,imaizumi2019deep,imaizumi2020advantage}. 
Also, since wavelets require $C^d$ basis functions with some $C>0$ to approximate functions with $d$-dimensional input (see \cite{cohen1993wavelets} for details), the computational complexity increases exponentially as $d$ increases, but DNNs can avoid this problem.
}

\textbf{Step (ii). Non-asymptotic Gaussian approximation}:
As a fundamental basis of this study, we develop a \textit{non-asymptotic Gaussian approximation} by applying an approximation scheme \citep{chernozhukov2013gaussian,chernozhukov2014gaussian}.
The scheme approximates an empirical stochastic process on an index set $\mF$ using the tractable supremum of the stochastic processes.
Importantly, the approximation scheme does not require a limit distribution of the empirical process.

We provide an overview of the scheme proposed by \cite{chernozhukov2014gaussian}.
Let $Z_\mF$ be a random variable by an empirical process
\begin{align}
    Z_\mF := \sup_{f \in \mF} \frac{1}{\sqrt{n}} \sum_{i=1}^n \{ f(X_i) - \Ep_\mu[f(X)]\}, \label{def:emp_proc}
\end{align}
with some set of functions $\mF$.
To approximate a distribution of $Z_\mF$, the scheme derives a stochastic process $\{G_\mF(f),f \in \mF\}$ with zero mean and known covariance.
Then, the scheme approximates the random variable $Z_\mF$ using a supremum of $G_\mF(f)$ on $\mF$.
Theorem 2.1 in \cite{chernozhukov2014gaussian} proves that
\begin{align*}
    \abs*{ Z_\mF - \sup_{f \in \mF}G_\mF(f) } \leq c_A n^{-1/6},
\end{align*}
holds for each $n$ with a constant $c_A > 0$ with high probability approaching to $1$.
Here, the approximator $\sup_{f \in \mF}G_\mF(f)$ does not depend on $\lim_{n \to \infty}Z_\mF$ (Note that $G_\mF(f)$ does \textit{not} have to be a limit of $Z_\mF$), hence the result holds regardless of the limit distribution of $Z_\mF$.

\begin{table*}[t]
  \label{sample-table}
  \centering
  \begin{tabular}{ccccccc}
    \toprule
    $W(\mu_n,\mu)$ &$\leftrightarrow$ &$\hat{W}(\mu_n,\mu)$&$\leftrightarrow$ &$\hat{Z}_W$ &$\leftrightarrow$ &$\tilde{Z}_W$\\
    \begin{tabular}{c} Wasserstein \\ distance \end{tabular} && \begin{tabular}{c} DNN \\ Approx. \eqref{def:wasserstein_dnn} \end{tabular}   && \begin{tabular}{c} Gaussian \\ Approx. \eqref{eq:sup_GP} \end{tabular}   && Algorithm \ref{alg:GMB}\\
    \bottomrule
  \end{tabular}
  \caption{List of the approximators. \label{table:list}}
\end{table*}

We extend the scheme and approximate the empirical {$1$-Wasserstein distance} using DNNs
\begin{align*}
    {\hat{W}_1(\mu_n,\mu)} = \sup_{f \in \Xi^{\lip,\gamma}(L,S)} \frac{1}{n}\sum_{i=1}^n \{ f(X_i) -\Ep_{\mu}[f(X)]\}.
\end{align*}
For the approximation, we derive a suitable stochastic process on $\Xi^{\lip,\gamma}(L,S)$.
We introduce a scaled Gaussian process $\{G_W(f):f \in \Xi^{\lip,\gamma}(L,S)\}$ whose mean is $0$ and its covariance function is $\bE_{X \sim \mu}[f(X)f'(X)] S^{-1} $ for $ f,f' \in \Xi^{\lip,\gamma}(L,S)$.
For any $f,f' \in \Xi^{\lip,\gamma}(L,S)$, $G_W(f)$ is a Gaussian random variable with its mean $\Ep[G_W(f)]=0$ and covariance $\Ep[G_{W}(f)G_{W}(f')]=\bE_\mu[f(X)f'(X)]S^{-1}$ using a scaling term $S^{-1}$.
Then, we define the following random variable
\begin{align}
    \hat{Z}_W := \sup_{f \in \Xi^{\lip,\gamma}(L,S)} G_W(f), \label{eq:sup_GP}
\end{align}
as an approximator for { $\hat{W}_1(\mu_n,\mu)$}.
Its validity is proved in the following lemma:
\begin{lemma}[Approximation for {$\hat{W}_1(\mu_n,\mu)$}]\label{lem:intermediateGAR}
Set $L=O(\log^2 n)$, and $S=S_n$ as $(S_n \log^6 n)/n \to 0$ as $n \to \infty$.
Then, we obtain
\begin{align*}
    &\Pr \left( \abs*{  \sqrt{\frac{n}{S}} {\hat{W}_1(\mu_n,\mu)} - \hat{Z}_W }  \leq \frac{c_W \sigma^2S^{1/6} \log n}{n^{1/6}} \right)\rightarrow 1, ~ (n \to \infty),
\end{align*}
where $\sigma^2 := \sup_{f,f' \in \Xi^{\lip,\gamma}(L,S)} \Ep_{X \sim \mu}[f(X)f'(X)]$ and $c_W > 0$ is an existing constant.
\end{lemma}
Lemma \ref{lem:intermediateGAR} states that $\hat{Z}_W$ is a random variable that behaves sufficiently close to {$\hat{W}_1(\mu_n,\mu)$} with probability approaching to $1$.
The setting $S=o(n \log^{-6}n)$ guarantees that the bound $\frac{c_W \sigma^2S^{1/6} \log n}{n^{1/6}}$ converges to $0$ as $n$ increases.
Note that the result holds regardless of the limit distribution of {$\hat{W}_1(\mu_n,\mu)$}. Hence we avoid the unavailability problem of limit distributions discussed in Section \ref{sec:related}.

{
Furthermore, we note that the convergence rate in $n$ is reasonably fast, because it decays polynomially in $n$ although it approximates a random element on the infinite-dimensional space $\lip(\mX)$.
An ordinary Gaussian approximation theory cannot obtain such a polynomial rate in $n$ because of the infinite dimensionality.
Especially, the Berry-Esseen bound \citep{bentkus2005lyapunov} for Gaussian approximation on a $p$-dimensional space is $O(p^{1/4}n^{-1/2})$, which diverges when $p=\infty$.
In contrast, the Gaussian approximation theory we use, which focuses only on approximating the maximum of Gaussian processes, has a valid error bound in $n$ even with the infinite dimension.
}

\textbf{Step (iii): Distributional approximation}:
We combine the previous results and evaluate how $\hat{Z}_W$ approximates the empirical Wasserstein distance {$W_1(\mu_n,\mu)$}.
We specify $S$ to balance the approximation error by DNNs and that of the non-asymptotic Gaussian approximation.
The result is presented in the following theorem:

\begin{theorem}[Approximation for {$W_1(\mu_n,\mu)$}]\label{thm:GAR_kolmogorov_2} 
Set $L = O(\log^2 n)$ and $S = S_n$ satisfying $n^{2d/(3+2d)}\log^{1/3}n / S_n \to 0$ and $ n  \log^{-6} n / S_n \to \infty$ as $n \to \infty$.
Then, we have
\begin{align*}
    &\Pr \left( \abs[\Big]{\sqrt{\frac{n}{S}} {W_1(\mu_n,\mu)} - \hat{Z}_W } \leq \frac{c'_W \sigma^2 S^{1/6} \log n}{n^{1/6}} \right)\rightarrow 1, ~ (n \to \infty),
\end{align*}
where $c'_W > 0$ is an existing constant.
\end{theorem}
Here, the multiplier $\sqrt{n/S}$ is a scaling sequence for {$W_1(\mu_n,\mu)$}.
Theorem \ref{thm:GAR_kolmogorov_2} states that $\hat{Z}_W$ approximates the randomness of the scaled {$W_1(\mu_n,\mu)$} with arbitrary accuracy as $n \to \infty$.
The result is valid without utilizing a limit distribution of {$W_1(\mu_n,\mu)$}, similar to Lemma \ref{lem:intermediateGAR}.
The convergence rate in Theorem \ref{thm:GAR_kolmogorov_2} is reasonably fast, because it is a polynomial rate in $n$ although we consider the infinite dimensional parameter $f \in \Xi^{\lip,\gamma}(L,S)$.

To obtain a more convenient formulation, we introduce an assumption on $\mu_n$.
\begin{assumption}[Anti-Concentration of {$W_1(\mu_n,\mu)$}] \label{asmp:cont_cdf}
    For $S = S_n$ as the setting in Theorem \ref{thm:GAR_kolmogorov_2} with any $n$, any $r \in \mathbb{R}$ and $\delta > 0$, there exists a constant $C_{r} > 0$ such that
    \begin{align*}
        \Pr(r \leq \sqrt{n/S} {W_1(\mu_n,\mu)} \leq r + \delta) \leq C_r \delta.
    \end{align*}
\end{assumption}
This assumption requires that the distribution of $\sqrt{n/S} {W_1(\mu_n,\mu)}$ does not degenerate at some specific point.
For example, when the distribution has a bounded density function, the assumption holds.
We numerically validate this assumption and prove that it empirically holds in Section \ref{sec:asmp1}.

{
With this assumption, we can measure the performance of the approximation by $\hat{Z}_W$ in terms of the Kolmogorov distance.
Formally, the Kolmogorov distance between random variables $X,X'$ is defined as $\sup_{t \in \mathbb{R}} | \Pr(X\leq t) - \Pr(X'\leq t) |$.
This distance is suitable for analyzing the hypothesis test, because this definition is compatible with a notion of confidence interval and rejection region.
Using the definition, we obtain the result as follows:
}
\begin{corollary} \label{cor:main}
    Suppose Assumption \ref{asmp:cont_cdf} holds.
    Then, with the settings of Theorem \ref{thm:GAR_kolmogorov_2}, we obtain
    \begin{align*}
        &\sup_{z \in \mathbb{R}} \abs*{ \Pr\left(\sqrt{\frac{n}{S_n}} {W_1(\mu_n,\mu)} \leq z \right) - \Pr\left(\hat{Z}_W \leq z \right)  } \to 0, ~ (n \to \infty).
    \end{align*}
\end{corollary}

\subsection{Algorithm: Multiplier Bootstrap}

We present an efficient algorithm to calculate the derived random variable $ \hat{Z}_W$.
Although $ \hat{Z}_W$ is tractable, computing a supremum of Gaussian processes is computationally difficult in general, hence the fast algorithm is profitable for practical applications. 

The proposed algorithm employs a \textit{multiplier bootstrap method} which generates a bootstrap Gaussian process with Gaussian random variable as multipliers \citep{chernozhukov2015comparison}.
We generate multiplier Gaussian variables $\{\xi_1,..., \xi_{n}\} \sim N (0,1)$ which are independent of $\mD_n$.
Then, we define a bootstrap process    
\begin{align}
     \tilde{Z}_W := \sup_{f \in \Xi^{\lip,\gamma}(L,S)}\sqrt{\frac{1}{nS_n}}\sum_{i=1}^{n} \xi_i \left( f(X_i) -  \hat{\Ep}_{\mD_n^X}[f(X)] \right).\label{def:btsp}
\end{align}

We summarize the multiplier bootstrap method in Algorithm \ref{alg:GMB}.
$T$ is a hyper-parameter for the iterations.

The convergence of $\tilde{Z}_W$ is provided by the following Theorem.
For further discussion, we introduce Assumption \ref{asmp:cont_cdf} and present our result in terms of the Kolmogorov distance.
\begin{theorem}[Validity of Gaussian Multiplier Bootstrap]\label{thm:GMB}
Suppose Assumption \ref{asmp:cont_cdf} holds.
Then, with the setting of Theorem \ref{thm:GAR_kolmogorov_2} and for each $n$, we obtain
\begin{align}\label{eq:GMB}
    &\sup_{z \in \mathbb{R}}\abs*{\Pr\left(\sqrt{\frac{n}{S_n}} {W_1(\mu_n,\mu)} \leq z\right) - \Pr\left(\tilde{Z}_W \leq z \mid \mD_n^X\right)} \to 0, ~ (n \to \infty).
\end{align}
\end{theorem}
Theorem \ref{thm:GMB} proves that the distribution of $\tilde{Z}_W$ with fixed $\mD_n^X$ can approximate the distribution of {$W_1(\mu_n,\mu)$} by random $\mD_n^X$.
We obtain the result by a combination of all the approximators developed in Section \ref{sec:approx}.
Table \ref{table:list} presents an overview of our strategy for approximating the distribution of {$W_1(\mu_n,\mu)$}.

From a computational aspect, Algorithm \ref{alg:GMB} requires a computational time of $O(n)$.
It is sufficiently faster than the empirical bootstrap and sub-sampling bootstrap methods, which require $O(n^2)$ computational time; they require $O(n)$ subsets for $O(n)$ sub-samples to achieve the same accuracy as that of our method.

\begin{algorithm}[H]

\caption{Gaussian Multiplier Bootstrap \label{alg:GMB}}
   
\begin{algorithmic}
    \REQUIRE $T \in \mathbb{N}, (L,S) \in \mathbb{N}^2,\mD_n^X$  
    
    \FOR{$t=1$ to $T$}
    \STATE Generate i.i.d. random variables $\{\xi_1^{(t)},..., \xi_{n}^{(t)}\} \sim N(0,1)$, independently of $\mD_n$.
    \STATE Compute a supremum of the following multiplier bootstrap process as
    \begin{align*}
        \tilde{Z}_{W}^{(t)} \leftarrow  \tilde{Z}_W \mbox{~as~}\eqref{def:btsp}.
    \end{align*}
    \ENDFOR
    \ENSURE $\{\tilde{Z}_W^{(1)},...,\tilde{Z}_W^{(T)}\}$ as a distribution of $\tilde{Z}_W$.
\end{algorithmic}
\end{algorithm}

{
\begin{remark}[Architecture Selection]
Our theory guarantees the existence of an appropriate neural network architecture, but does not specify it. 
Instead, we conduct compprehensive experiments in Section \ref{sec:select_architecture} in Appendix to provide suggestions for architectural choices.
From the result, we recommend increasing the dimension of each layer (with) and limiting the number of hidden layers. 
Also, a wide variety of network types are allowed, such as residual networks \cite{he2016deep}.
\end{remark}

}

    \section{Applications to Hypothesis Test}

    We develop one- and two-sample tests by applying the distributional approximation.
    In the following, we set $\alpha \in (0,1)$ as an arbitrary significance level.

    \textbf{One-sample test}:
    We test the probability measure $\mu$, which generates $\mD_n^X$ and is identical to a pre-specified measure $\mu_0$.
    We consider the null/alternative hypotheses as follows:
    \begin{align*}
        H_0: \mu = \mu_0, \qquad H_1: \mu \neq \mu_0.
    \end{align*}
    To test the hypothesis, we use {the empirical $1$-Wasserstein distance} $\sqrt{n/S}{W_1(\mu_n, \mu_0)}$ as the test statistic.
    The asymptotic distribution of the test statistics can be approximated by the Gaussian multiplier bootstrap by Algorithm \ref{alg:GMB}. 
    We present the procedure of the one-sample test as follows:
    \begin{enumerate}
        \item We calculate the distribution of $\tilde{Z}_W$ and its $(1-\alpha)$-quantile $\tilde{q}(1-\alpha)$.
        \item If $\sqrt{n/S}{W_1(\mu_n,\mu_0)} \geq \tilde{q}(1-\alpha)$ holds, reject $H_0$. Otherwise, accept $H_0$.
    \end{enumerate}
    The validity of the test is presented in the following theorem, which shows that the Type I error converges to the specified $\alpha$.
    \begin{theorem}\label{thm:one_test}
    Suppose Assumption \ref{asmp:cont_cdf} holds, $S_n$ satisfies the setting of Theorem \ref{thm:GAR_kolmogorov_2}.
    Under the null hypothesis $H_0:\mu=\mu_0$, we have the Type I error as
    \begin{align*}
    &\Pr\left(\sqrt{\frac{n}{S_n}} {W_1(\mu_n,\mu_0)} >  \tilde{q}(1-\alpha) \right) = \alpha + o(1), ~~(n \to \infty).
    \end{align*}
    \end{theorem}
    
    Additionally, we can construct a confidence interval that contains a true Wasserstein distance with an arbitrary significance level.
    Namely, an interval
        $\left[ \frac{\sqrt{S} \tilde{q}(\alpha/2)}{\sqrt{n}}, \frac{\sqrt{S} \tilde{q}(1-\alpha/2)}{\sqrt{n}}\right]$
    contains {$W_1(\mu,\mu_0)$} with probability $(1-\alpha)$ asymptotically.
   
    \textbf{Two-sample test}:
    We consider a two-sample test to investigate the discrepancy between two probability measures from two sets of observations.
    In addition to $\mD_n^X \sim \mu$, suppose that we observe a set of independent $m$ observations $\mD_m^Y := \{Y_1,...,Y_m\} $ from another distribution $\nu$.
    We define $\nu_m := m^{-1} \sum_{j=1}^m \delta_{Y_j}$ as its empirical distribution.
    Let $\tilde{Z}_{W}^{(X)}$ be an output of Algorithm \ref{alg:GMB} from $\mD_n^X$, and $\tilde{Z}_{W}^{(Y)}$ be an output from $\mD_n^Y$.
    To calculate $\tilde{Z}_{W}^{(Y)}$, we use a DNN with $L$ layers and $S_m$ parameters.
    We define $\rho_{n,m} = (nm)/(m+n)$ and $\lambda_{n,m} := m/(n+m)$.

    In the two-sample test, we consider the following null/alternative hypotheses:
    \[
    H_0: \mu = \nu, \qquad H_1: \mu \neq \nu.
    \]
    To investigates the hypotheses, we employ $\sqrt{\rho_{n,m}/(S_n \wedge S_m)}{W_1(\mu_n, \nu_m)}$  as a test statistic.
    We define $\breve{q}_X(1-\alpha)$ and $\breve{q}_Y(1-\alpha)$ as $(1-\alpha)$-quantiles of the distributions $\sqrt{\lambda_{n,m}}\tilde{Z}_{W}^{(X)}$ and $\sqrt{1-\lambda_{n,m}}\tilde{Z}_{W}^{(Y)}$, and also define $\breve{q}(1-\alpha):=\inf_{r \in (0,1)}\breve{q}_X(1-r\alpha) +  \breve{q}_Y(1-(1-r)\alpha)$.
    The procedure of the two-sample test is as follows:
    \begin{enumerate}
        \item We calculate $\breve{q}(1-\alpha)$ from $\mD_n^X$ and $\mD_m^Y$.
        \item If $\sqrt{\rho_{n,m}/(S_n \wedge S_m)}{W_1(\mu_n,\nu_m)} \geq \breve{q}(1-\alpha)$ holds, reject $H_0$. Otherwise, accept $H_0$.
    \end{enumerate}
    The following theorem presents the validity of the test by proving the convergence of the Type I error.
    \begin{theorem}\label{thm:two_test}
    Suppose $n/m \to \Bar{\lambda} \in (0,1)$ as $n,m \to \infty$.
    Also, suppose Assumption \ref{asmp:cont_cdf} holds, and both $S_n$ and $S_m$ satisfy the setting of Theorem \ref{thm:GAR_kolmogorov_2}.
    Then, under the null hypothesis $H_0:\mu=\nu$, we have the Type I error as
    \begin{align}\label{eq:null_one_append}
        &\Pr \left(\sqrt{\frac{\rho_{n,m}}{S_n \wedge S_m}} {W_1(\mu_n,\nu_n)} > \breve{q}(1-\alpha) \right) \leq \alpha +o(1), ~~(n,m \to \infty).
    \end{align}
    \end{theorem}
    We note that the proposed two-sample test is conservative unlike the one-sample test, due to the triangle inequality in measuring the distribution of {$W_1(\mu_n,\nu_n)$}.
    Furthermore, similar to the one-sample test, we can develop an asymptotically valid confidence interval.

\section{Experiments}\label{sec:pre_exp}
    We conduct experiments to validate the performance of the proposed tests with the Wasserstein distance. 
    
    \subsection{Setting for the proposed test}
    For the proposed test, we utilize a fully-connected DNN with the ReLU activation and $3$ layers, where each layer contains $100$ or $200$ nodes.
    To calculate the supremum in \eqref{def:btsp}, 
        we employ Vanilla SGD or ADAM \citep{adam_opt} optimizer.
    Additionally, we use a \textit{spectral normalization} \citep{miyato2018spectral} to restrict the Lipschitz constants of the functions to $1$.

\subsection{Validation of Distributional Approximation} \label{sec:exp_validation} 
    
\begin{figure}[htbp]
    \centering
    \begin{minipage}{28mm}
        \centering
        \includegraphics[width=28mm]{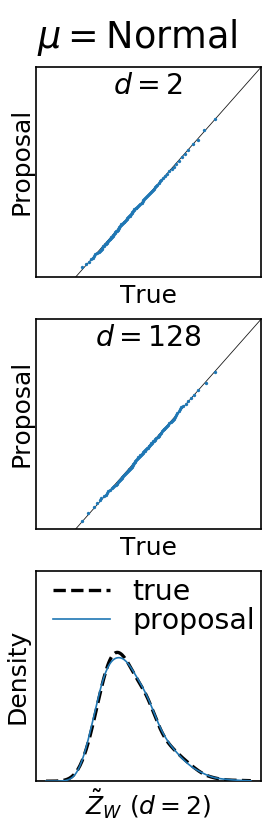}
    \end{minipage}
    \begin{minipage}{28mm}
        \centering
        \includegraphics[width=28mm]{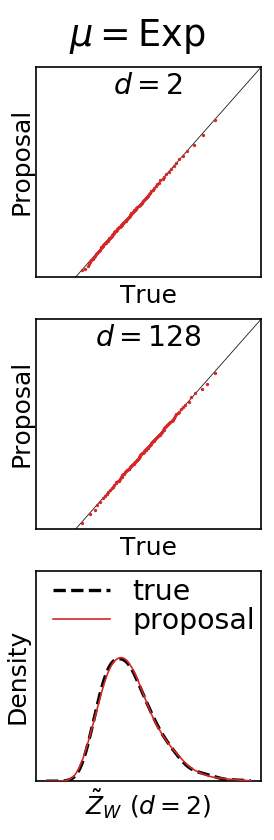}
    \end{minipage}
    \begin{minipage}{28mm}
        \centering
        \includegraphics[width=28mm]{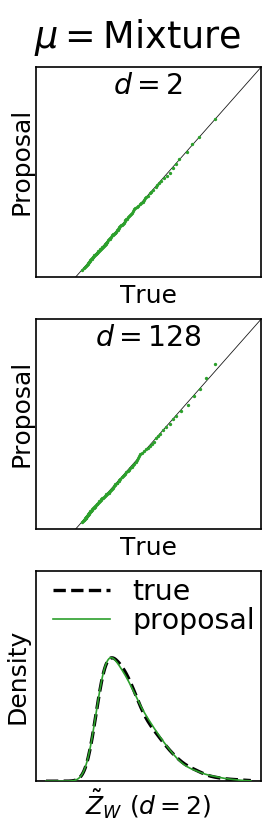}
    \end{minipage}
    \caption{Upper and middle rows show the Q-Q plots for the true distribution (vertical) and a distribution of the proposed $\tilde{Z}_W$ (horizontal), using $2000$ $d$-dimensional samples from the normal, exponential, and mixture of Gaussian distributions. The lower row indicates the densities of the true distribution and the distribution of the proposal with $2$-dimensional cases from corresponding distributions in their columns. }
    \label{img:QQ}
\end{figure}

    We validate the performance of our proposed method in Algorithm \ref{alg:GMB} which approximates the true distribution of {$W_1(\mu_n,\mu)$}.
    To generate a true distribution, we set $1,000$ and generate $\mD_n^X$ and then calculate {$W_1(\mu_n,\mu)$}.
    Then, we repeat this procedure for $2,000$ times, hence we use $1,000\times 2,000$ observations in total.
    To validate the proposed method, we draw $n = 1,000$ samples from $\mu$ and calculate $\tilde{Z}_W$. 
    
    As configurations of $\mu$, we generate $d$-dimensional random vectors ($d\in \{2,128\}$) whose elements independently follow 
        (i) a normal $\mathcal{N}(0,1)$,
        (ii) an exponential $\mathcal{E}(1)$,
        and
        (iii) a mixture of Gaussian $\frac{1}{2}\mathcal{N}(-4,1)+\frac{1}{2}\mathcal{N}(4,1)$.
    For the proposed distribution $\tilde{Z}_{W}$, 
        we set $T=2,000$.

    We present the Q-Q plots and histograms in Figure~\ref{img:QQ}.
    The results show that our proposed method can approximate the distribution of the {empirical $1$-Wasserstein distance} well in every configuration of $\mu$.

    \subsection{Comparison with Other Hypothesis Tests}

\begin{figure}[htbp]
    \centering
\includegraphics[width=0.98\hsize]{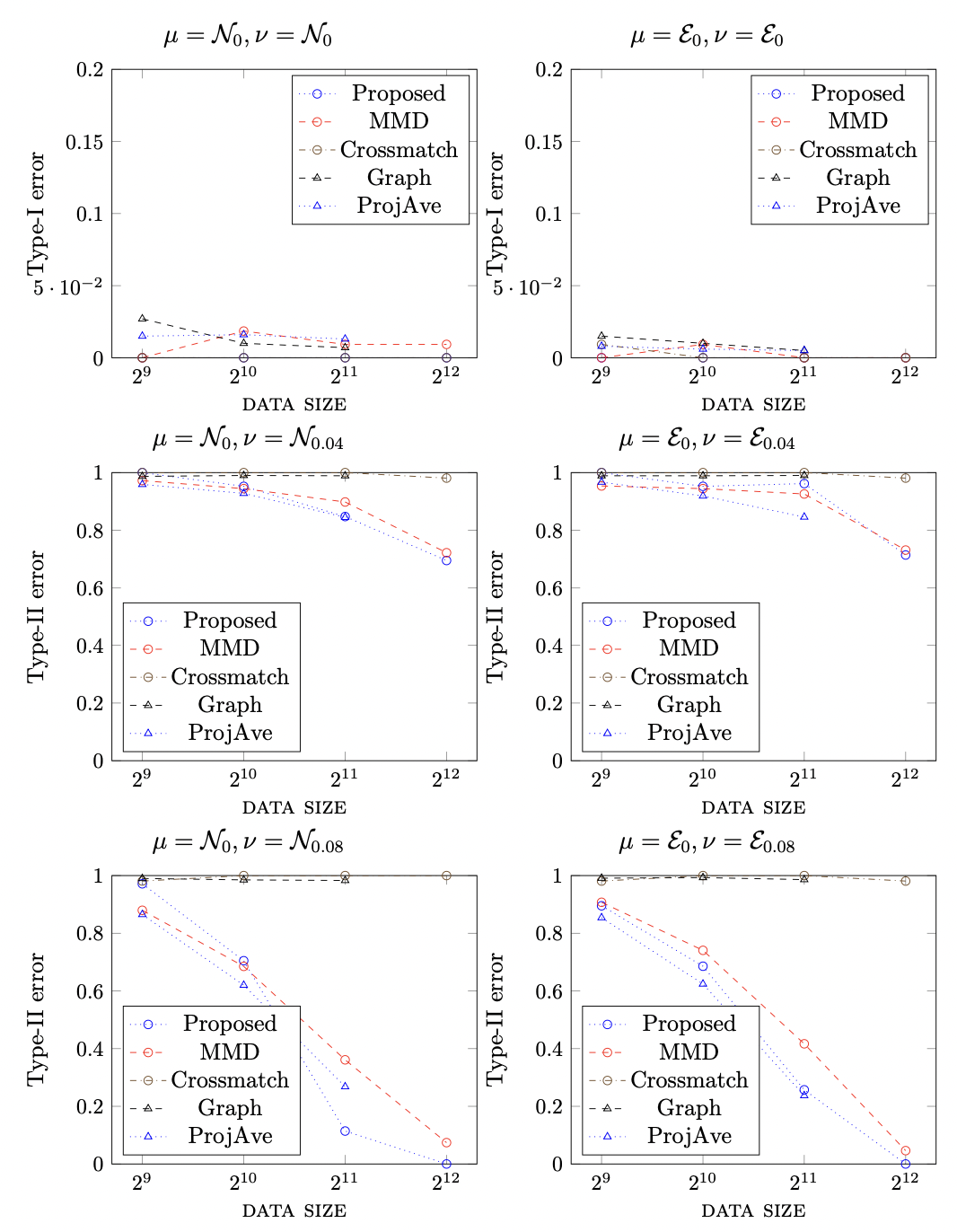}
\caption{Type I errors (upper row) and type II errors (middle and lower rows) of the proposed method and the baselines. We set $\alpha=0.01$, and define the Gaussian distribution $\mathcal{N}_{\upsilon}$ (left column) and  exponential distributions $\mathcal{E}_{\lambda}$ (right column).  \label{fig:compare}}
\end{figure}

    We compare the performance of the proposed test with other tests by reporting the type I errors, type II errors, and computational times.
    As the methods for comparison, we consider the maximum mean discrepancy (MMD) test \citep{gretton2012kernel}, the crossmatch test \citep{rosenbaum2005exact}, the graph-based two-sample test \cite{chen2017new}, and the projection-averaging test \cite{kim2020robust}.
    We set $\alpha = 0.01$ and $n \in \{512, 1024, 2048, 4096\}$, and then generate $n$ samples from two distributions with different parameters.
    First, we consider two-dimensional Gaussian distributions, namely, we set 
    $\mu = \mathcal{N}_{\upsilon}=\mN(\upsilon\mathbf{1}, I)$ with $\upsilon \in \{0, 0.08\}$, and $\nu = \mathcal{N}_{0}$.
    $\mathbf{1}$ and $I$ represent an all one vector and an identity matrix, respectively.
    Second, we consider two products of exponential distributions, namely,  $\mu=\mathcal{E}_{\lambda}= \Exp(1+\lambda)\otimes \Exp(1+\lambda)$, with $ \lambda \in \{0, 0.08\}$, and $\nu=\mathcal{E}_0$.

    We plot the errors in Table~\ref{fig:compare}.
    For the type II errors, the proposed Wasserstein test works slightly better than the other methods as $n$ increases.
    {The results of the projection-averaging and the graph-based two-sample test with $n=4096$ are not shown, because the amount of used memory was too large to be calculated in our computing environment \footnote{We used a server for computation provided by \textit{Amazon Web Service} with $6$ cores and $64$ GB of memory.}. These methods require multiple $n \times n$ distance matrices and generate new graphs of matrices internally, making them difficult to compute when the number of data is large.
    The result in Table \ref{fig:compare} shows that the proposed test works stably with large sample sizes, and also can distinguish the two close distributions because the Wasserstein distance induces a weak topology.
    }
    
    We also show computational time of a part of the methods in Figure \ref{fig:time}.
    {Our method is not the fastest, but it works within a reasonable amount of time, which is almost constant, even as the number of data increases. 
    Other methods increase the computation time with respect to the number of data. 
    This is because the proposed test requires a computational time of $O(n)$, while some other methods, such as the MMD test, require $O(n^2)$ computational time.
    Also, the projection averaging and the graph-based test become impossible to compute when the number of data is large ($n=4096$), for the reasons mentioned earlier. 
    These results suggest that our proposed method is a method that works stably even with a large number of data.
    }
    
\begin{table}[htbp]
    \centering
    \begin{tabular}{c cccc}
        \toprule
        data size   &512 & 1024 & 2048 & 4096 \\ 
        \midrule
        Proposed        & 149.6        & 151.5       & 151.9   &   157.2    \\
        MMD        & 15.8    & 87.0       & 369.8      & 1428.4   \\
        Crossmatch            & 2.7    & 4.9  & 14.2     & 65.2    \\
        Graph    &   0.3   & 0.9    &  3.0  &N/A       \\
        ProjAve    &  4.4  & 21.4  & 89.4 & N/A      \\
        \bottomrule
    \end{tabular}
    \caption{{Computational time (seconds) against the data size by the proposed method and the baselines. The values are mean of $50$ repetitions.\label{fig:time}}}
\end{table}

    \subsection{Comparison with Divergences under Singular Measures} \label{sec:exp_singular}
    We demonstrate that the proposed test exhibits better performance that other divergence-based tests, when $\mu$ and $\nu$ are singular (i.e. there exists $\Omega \subset \mX$ such that $\mu(\Omega) = 0$ and $\nu(\Omega) \neq 0$ or vise versa).
    It is well known that common divergences, such as Kullback-Leibler divergence, cannot work well with singular measures, although singular measures frequently appear in tasks involving generative adversarial networks (GANs).
    We compare the proposed Wasserstein-based test with the Kullback-Leibler divergence (KL-div), Jensen–Shannon divergence (JS-div), and total-variation (TV).
    
    \begin{figure}[H]
    \centering
        \begin{minipage}{47mm}
            \includegraphics[width=1.0\hsize]{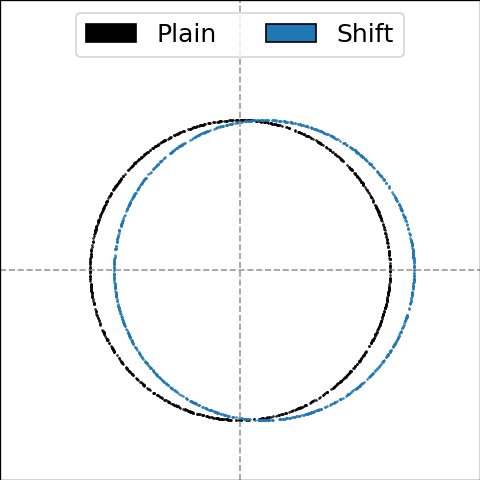}
        \end{minipage}
        \begin{minipage}{47mm}
            \includegraphics[width=1.0\hsize]{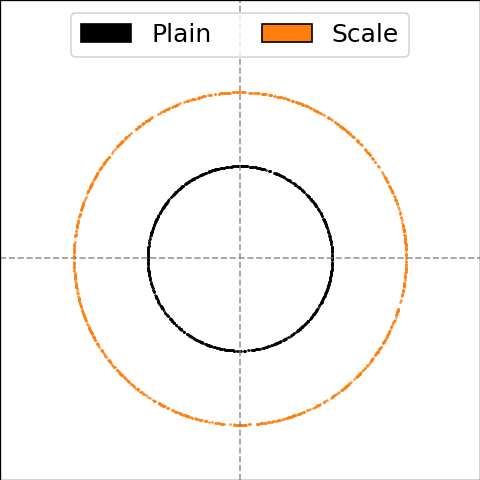}
        \end{minipage}
        \caption{Support of the measure in the $2$-dimensional space.
        The black circles is for the measure \textit{Plain}, the blue circle is \textit{Shift}, and the orange circle is \textit{Scale}.}
        \label{fig:manifolds}
\end{figure}

    We consider three configurations for $\mu$ and $\nu$, namely, \textit{Plain}, \textit{Shift}, and \textit{Scale}, as follow.
    Let \textit{Plain} be a uniform measure on a set $\{(x_1,x_2)\mid x_1^2+x_2^2=\frac{1}{4}\}$, \textit{Shift} be a uniform measure on a set $\{(x_1,x_2): (x_1-0.08)^2+x_2^2=\frac{1}{4}\}$, and \textit{Scale} be a uniform measure on a set $\{(x_1,x_2) \mid x_1^2+x_2^2=\frac{1}{4}1.8^2\}$.
    We note that the measures are singular to each other, since their supports have few intersections, as shown in Figure~\ref{fig:manifolds}.
    We generate $n=1024$ samples.

    For comparison, we conduct a permutation test with the divergences.
    To calculate the divergences, we use a density function of $\mu$ and $\nu$ estimated by a kernel method with the Gaussian kernel.
    Its bandwidth is selected by cross-validation.
    We conduct $20$ trials with a significance level of $\alpha=0.1$.

    We present the results in Table~\ref{table:topology}.
    The KL- and JS-divergences do not capture the difference between the measures, and the other divergences fail in some cases.
    In contrast, the proposed test with the Wasserstein distance works well although the supports do not overlap completely.
    
    \begin{table}[htbp]
    \centering
    \begin{tabular}{c ccc}
        \toprule
        $\mu/\nu$   & \textit{Plain}$/$\textit{Plain} & \textit{Plain}$/$\textit{Shift} & \textit{Plain}$/$\textit{Scale} \\ 
        \midrule
        KL-div        & 5        & 100       & 100         \\
        JS-div        & 15    & 100       & 100          \\
        TV            & 10    & \bf{0}  & \bf{0}         \\
        \midrule
        Proposed    & \bf{0}    & \bf{0}  & \bf{0}       \\
        \bottomrule
    \end{tabular}
    \caption{Error-rate over $20$ trials of the one-sample test with singular probability measure. 
    The bold letters indicate good scores.}
    \label{table:topology}
\end{table}

    \subsection{Real Data Analysis}

    \textbf{Molecular Data:}
    We evaluate the proposed test and the other tests using bioinformatic datasets, named Leukemia and p53, from the Molecular Signatures Database (MSigDB)
        \footnote{\url{http://software.broadinstitute.org/gsea/index.jsp}}\citep{liberzon2011molecular}.
    Each of the datasets contains from $8$ to $33$ samples and around $10,000$ features.
    The datasets contain two different groups such as male and female.
    We conduct the two sample-test with the following two settings.
    First, we consider $\mu \neq \nu$ by setting $\mu$ as a distribution for males and $\nu$ for females.
    For the $\mu=\nu$ settings, we randomly split each of the groups, labeled as Leukemia-ALL, Leukemia-AML, p53-MUT, and p53-WT, into two groups, respectively.

\begin{table*}[htbp]
    \centering
    \begin{tabular}{c cc ccc}
        \toprule
            \multicolumn{1}{l}{} & 
            \multicolumn{2}{c}{Test with Different Groups} & 
            \multicolumn{3}{c}{Test with  Same Groups} \\ 
            \cmidrule(r){2-3} \cmidrule(lr){4-6} 
                & Leukemia    & p53    &  Leukemia-AML    & p53-MUT   & p53-WT    \\ 
\midrule
Sample sizes    & 24-24        & 33-17    
                & 12-12     & 17-16        & 8-9  \\
Feature sizes   & 10056        & 10095    
                & 10056     & 10095        & 10095  \\
\midrule
KL-div             & \bf{Reject}    & \bf{Reject}    
                & Reject        & \bf{Accept}   & Reject \\
JS-div             & Accept        & Accept    
                & \bf{Accept}    & \bf{Accept}   & \bf{Accept} \\
MMD             & \bf{Reject}    & Accept    
                & \bf{Accept}    & \bf{Accept}    & Reject \\
\midrule
Proposed        & \bf{Reject}    & Accept    
                & \bf{Accept}    & \bf{Accept}   & \bf{Accept} \\
        \bottomrule
    \end{tabular}
    \caption{Results of two sample tests using the  proposed method and the other methods on MSigDB. Bold font indicates the correct decision. \label{tab:bio}}
    \vspace{-3mm}
\end{table*}

    Table~\ref{tab:bio} presents the result.
    The proposed test specifies the differences correctly with the same group datasets.
    In contrast, other methods fail to find the difference with the Leukem or p53-WT datasets.

    \textbf{Handwritten Letter Data (MNIST):}
    To show the interpretability of the proposed test, we conduct a two-sample test for distinguishing between the hand-written letter images from the Modified National Institute of Standards and Technology database (MNIST) dataset \cite{lecun1998gradient}.
    The dataset contains images with $d=576$ pixels for a hand-written number from $0,1,...,9$.
    We set $n=200$ and $n=m$, and then set $\mu_n$ and $\nu_m$ as sampled images of a letter $1$, $3$, or $5$ from the dataset.
    Then, our proposed test investigates whether $\mu_n$ and $\nu_m$ represent the same number.
    We repeat the sampling and test fo $30$ times.

    Figure~\ref{fig:rejection_region} plots pairs of a critical value and a test statistics for each of the repetitions.
    The results show the relative difficulty in distinguishing between the images of $3$ and $5$, rather than $1$ and $5$.
    This is consistent with the intuition from the hand-written images, and the Wasserstein distance succeeds in capturing the intuition.

\begin{figure}[H]
    \centering
    \begin{minipage}{1.0\hsize}
        \centering
        \includegraphics[width=0.6\hsize]{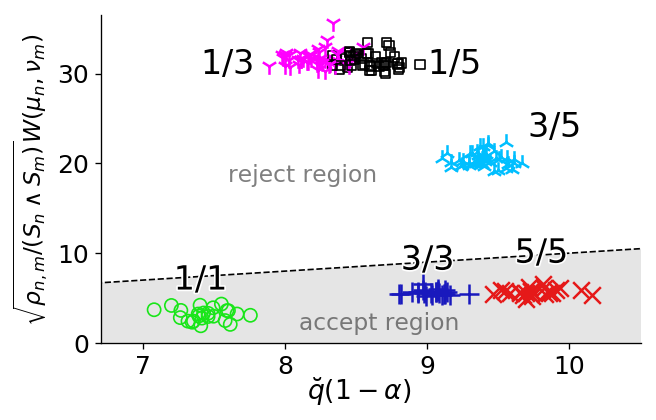}
        \captionsetup{width=1.0\textwidth}
        \caption{Plots of the test statistics and the critical value for each repetition.
        }
        \label{fig:rejection_region}
    \end{minipage}
\end{figure}

    \section{Conclusion}
    We developed a hypothesis test with {the $1$-Wasserstein distance}.
    We utilized the non-asymptotic Gaussian approximation and developed a valid hypothesis test, then solve the problem of the unavailable limit distribution of {the empirical $1$-Wasserstein distance}. 
    For practical applications, we developed an approximation using DNNs and a multiplier bootstrap method.
    Furthermore, we validated our method through experiments, and demonstrated that it performs better than the other tests owing to the weak topological property of the Wasserstein distance.

\appendix

\section{Supportive results}\label{app:support}

We provide additional notations.
For a function $f:\mX \to \mathbb{R}$ and $q \in (0,\infty]$, $\|f\|_{L^p} := (\int_{\mX} f^p(x) dx)^{1/p}$ be an $L^p$-norm.
For a measure $Q$ on $\mX$, $\|f\|_{L^p(Q)} := (\int_{\mX} f^p(x) dQ(x))^{1/p}$ is an $L^p$-norm associated with $Q$.
If $f$ has a different domain $\Omega$, we write $\|f\|_{L^p(\Omega)} := (\int_{\Omega} f^p(x) dx)^{1/p}$.
For a set $\Omega$ associated with a norm $\|\cdot\|$, we define $N(\varepsilon, \Omega, \|\cdot\|) := \min\{N \mid \{\omega_j\}_{j=1}^N \subset \Omega, \cup_{j=1}^N \{\omega \mid \|\omega_j - \omega\|\leq \varepsilon \} \supset \Omega\}$ as a covering number of $\Omega$ with $\varepsilon > 0$.

In this section, we present supportive theorems and lemmas.
First, we provide the following theorem, which is an adapted version of the result for the non-asymptotic approximation.

\begin{theorem}[{Corollary 2.2 in \cite{chernozhukov2014gaussian}, adapted to our setting}] \label{thm:CCK}
    Let $Z_\mF$ be a supremum of an empirical process as \eqref{def:emp_proc}.
    Also, let $\mF$ be a functional class wich an envelope function $F$ such that $\sup_Q N(\varepsilon\|F\|_{L^2}, \mF, \|\cdot\|_{L^2(Q)}) \leq (A/\varepsilon)^v$ with existing constants $A,v > 0$ for all $\varepsilon \in (0,1)$.
    Suppose for some $b > 0$ and $q \in [4,\infty]$, $\sup_{f \in \mF}\Ep_{\mu}[f(X)^k] \leq \sigma^2b^{k-2}$ for $k=2,3,4$ and $\Ep_\mu[|F(X)|^q]^{1/q} \leq b$ holds.
    Then, for every $\gamma \in (0,1)$, there exists a random variable $\tilde{Z}_\mF := \sup_{f \in \mF} G_P(f)$ such that
    \begin{align*}
        &\Pr \left\{ \abs{Z_\mF - \tilde{Z}_\mF} > \frac{b K_n}{\gamma^{1/2}n^{1/2 - 1/q}} + \frac{(b\sigma)^{1/2}K_n^{3/4}}{\gamma^{1/2}n^{1/4}} + \frac{(b\sigma^2 K_n^2)^{1/3}}{\gamma^{1/3}n^{1/6}} \right\} \leq C\left(\gamma + \frac{\log n}{n}\right),
    \end{align*}
    where $K_n = cv(\log n \vee \log (Ab/\sigma))$ and $c,C>0$ are constants depends on only $q$.
\end{theorem}

Our Lemma \ref{lem:intermediateGAR} is built on Theorem 2.1 in \cite{chernozhukov2014gaussian} and it requires pointwise measurability on a function class. The function class $\mF$ is said to be pointwise measurable if there exists a countable subclass $\mathcal{G} \subset \mF$ such that for every $f \in \mF$ there exists a sequence $g_m \in \mathcal{G}$ with $g_m \to f$ pointwise; See Section 2.3 in \cite{vdVW1996} for instance. The pointwise measurablity ensures a supremum of empirical process is measurable map from a sample space $\Omega$ to $\mathbb{R}$. 
    
Note that a collection of functions that is separable for the supremum norm satisfies the pointwise measurability (cf. Example 2.3.4 in \cite{vdVW1996}). Since bounded Lipschitz function classes with its totally bounded domain are separable with respect to the supremum norm, the function class $\Xi^{\lip,\gamma}(L, S)$ must be pointwise measurable; see Theorem 11.2.4 and Corollary 11.2.5 in \cite{dudley_2002} for details.

Second, we present several results for a covering number of a set of DNNs and their integration.
These results will be used in the condition of the non-asymptotic approximation theorem.

\begin{lemma}[$L^2$-entropy number of $\Xi(\mX)$, Lemma 5 in \cite{schmidt2017nonparametric}]\label{lem:entropy}
Given a neural network, let $V$ be a product of a number of nodes in $L$ layers. 
For any $\varepsilon> 0$, we have 
\begin{equation}\label{eq:entropy}
    \log N(\varepsilon, \Xi(L,S), \|\cdot\|_{L^2(Q)}) \leq (S+1) \log\left(\frac{4(L+1)V^2}{\varepsilon} \right), 
    \end{equation}
    where $Q$ is any finite measure on $(\mX, \mA)$.
\end{lemma}

\begin{lemma}[Uniform entropy integral for $\Xi^{\lip,\gamma}(L,S)$]\label{lem:entropyint}
    Let $F_S$ be a envelope function of $\Xi^{\lip,\gamma}(L,S)$, that is, $F_S(x) \geq |f(x)|$ holds for all $f \in \Xi^{\lip,\gamma}(L,S)$ and $x \in \mX$. 
    For each integer $S$, define the uniform entropy integral as
    \begin{align*}
    J_s(\delta)&:=J(\delta, \Xi^{\lip,\gamma}(L,S), F_S) \\
    &:= \int_0^\delta \sup_Q \sqrt{1+\log N(\varepsilon\|F_S\|_{L^2(Q)}, \Xi^{\lip,\gamma}(L,S), \|\cdot\|_{L^2(Q)} )}d\varepsilon,
    \end{align*}
    where the supremum is taken over all finitely discrete probability measures on $(\mX, \mA)$.
    Then, for each integer $S$, there exists a global constant $C$ such that 
    \begin{equation}
    J_S(\delta)  \le C \sqrt{S} \sqrt{\delta^2 \log (1/\delta)}.
    \end{equation}
\end{lemma}

Third, we provide theoretical results for the expressive power of DNNs.
We will use these results for the approximation for the Wasserstein distance by DNNs.

\begin{lemma}[Theorem 5 in \cite{schmidt2017nonparametric}, adapted to our setting]  \label{thm:sh}
    For any function $f \in \lip(\mX)$ and any integers $m \geq 1$ and $N \geq 2^d \vee 2e^d$, there exists a deep neural network $\tilde{f}$ with parameters bounded by $1$, depth $L=8+(m+5)(1+ \log_2d)$ and a number of parameters $S \leq 141(d+2)^{3+d}N(m+6)$ such that
    {
    \begin{align*}
        \|\tilde{f} - f\|_{L^\infty} \leq 3(2+d^2)6^d N 2^{-m} + 3N^{-1/d}.
    \end{align*}
    }
\end{lemma}
We obtain this result by substituting $\beta = 1$, $r=d$, and $K=1$ in \citep{schmidt2017nonparametric}.
Based on the result, we obtain the following simplified result:
\begin{theorem}\label{thm:NN-approx}
    For any function $f^* \in \lip(\mX)$, there exists a constant $c' > 0$ such that for any $\varepsilon \in (0,1/2)$, there is a neural-network $f \in \Xi(L,S)$ with $ L=O((1 + \log S) (1+\log d))$ layers and all its parameters are bounded by $1$, such as 
    \begin{align*}
        \|f^*-f\|_{L^\infty} \leq c' S^{-1/d}.
    \end{align*}
\end{theorem}
We omit its proof because it is a simple application of Lemma \ref{thm:sh}.

\section{Proofs}\label{app:proofs}

\subsection{Proof of Lemma \ref{lem:NN-approx}}

Fix $\mu_1, \mu_2$ and sufficiently large $S$.
Also, set $L=O((1+ \log S) (1+ \log d))$.
Throughout this proof, the existence of the supremum in $\lip(\mX)$ is assured by the fundamental property of the dual form of the Wasserstein distance (e.g., Theorem 5.10 in \cite{villani2008optimal}).
Also, the supremum in $\Xi^{\lip,\gamma}(L, S)$ is also assured owing to the compactness of the parameter space and the continuity of the activation function.

First, we show that {$W_1(\mu_1,\mu_2) - \hat{W}_1(\mu_1,\mu_2) \leq c S^{-1/d}$}.
We have
\begin{align*}
    &{W_1(\mu_1,\mu_2) - \hat{W}_1(\mu_1,\mu_2)}\\
    &=\sup_{f \in \lip(\mX)}\{ \Ep_{\mu_1}[f(X)] - \Ep_{\mu_2}[f(X)]\} - \sup_{f \in \Xi^{\lip,\gamma}(L,S)}\{ \Ep_{\mu_1}[f(X)] - \Ep_{\mu_2}[f(X)]\}\\
    &=( \Ep_{\mu_1}[f^*(X)] - \Ep_{\mu_2}[f^*(X)]) - \sup_{f \in \Xi^{\lip,\gamma}(L,S)}\{ \Ep_{\mu_1}[f(X)] - \Ep_{\mu_2}[f(X)]\},
\end{align*}
where $f^*$ is the supremum in $\lip(\mX)$.

By Lemma \ref{lem:approx_lip}, there exists $\hat{f} \in \Xi(L,S)$ with $\gamma = O(S^{-1/d})$ which satisfies
\begin{align*}
    \|f^* - \hat{f}\|_{L^\infty} \leq c'S^{-1/d},
\end{align*}
where $c' > 0$ is an existing constant.
By using $\hat{f} \in \Xi^{\lip,\gamma}(L,S)$, we continue the inequality as
\begin{align*}
    &( \Ep_{\mu_1}[f^*(X)] - \Ep_{\mu_2}[f^*(X)]) - \sup_{f \in \Xi^{\lip,\gamma}(L,S)}\{ \Ep_{\mu_1}[f(X)] - \Ep_{\mu_2}[f(X)]\}\\
    &\leq  \{(\Ep_{\mu_1}[f^*(X)] - \Ep_{\mu_2}[f^*(X)]) -  (\Ep_{\mu_1}[\hat{f}(X)] - \Ep_{\mu_2}[\hat{f}(X)])\}\\
    &= \{\Ep_{\mu_1}[f^*(X)] - \Ep_{\mu_1}[\hat{f}(X)] \} + \{\Ep_{\mu_2}[\hat{f}(X)] - \Ep_{\mu_2}[f^*(X)] \}.
\end{align*}

Also, since $\mu_1$ is a probability measure, the H\"older's inequality provides
\begin{align*}
    \Ep_{\mu_1}[f^*(X)] - \Ep_{\mu_1}[\hat{f}(X)] =  \int_{\mX} (f^* - \hat{f})d\mu_1 \leq \|f^* - \hat{f}\|_{L^\infty}  \mu_1(\mX) = \|f^* - \hat{f}\|_{L^\infty}.
\end{align*}
For the term with $\mu_2$, we obtain the same inequality respectively.
Hence, we obtain that 
\begin{align}
    {W_1(\mu_1,\mu_2) - \hat{W}_1(\mu_1,\mu_2) \leq 2c'S^{-1/d}.} \label{ineq:ww1}
\end{align}

Second, we show {$ \hat{W}_1(\mu_1,\mu_2) - W_1(\mu_1,\mu_2)  \leq \gamma$}.
We evaluate the value as
\begin{align}
    &{\hat{W}_1(\mu_1,\mu_2) - W_1(\mu_1,\mu_2)} \notag \\
    &= \sup_{f \in \Xi^{\lip,\gamma}(L,S)}\{\Ep_{\mu_1}[f(X)] - \Ep_{\mu_2}[f(X)]\} - \sup_{f \in \lip(\mX)}\{\Ep_{\mu_1}[f(X)] - \Ep_{\mu_2}[f(X)]\} \notag \\
    & = \sup_{f' \in \Xi^{\lip,\gamma}(L,S)} \inf_{f \in \lip(\mX)}\{\Ep_{\mu_1}[f(X) - f'(X)] - \Ep_{\mu_2}[f(X) - f'(X)]\} \notag \\
    &= \sup_{f' \in \Xi^{\lip,\gamma}(L,S)} \inf_{f \in \lip(\mX)} \int (f(x) - f'(x)) d(\mu_1 - \mu_2)(x). \label{ineq:ww}
\end{align}

To bound the term, we consider a minimax distance $ \sup_{f' \in \Xi^{\lip,\gamma}(L,S)} \inf_{f \in \lip(\mX)} \|f - f'\|_{L^\infty}$.
Fix $f' \in \Xi^{\lip,\gamma}(L,S)$ arbitrary.
Then, we consider $f \in \lip(\mX)$ such that $f(\Bar{x}) = f'(\Bar{x})$ for some $\bar{x} \in \mX$, then define $f(x)$ and for $x \in \mX \backslash \{\Bar{x}\}$ as
\begin{align*}
    f(x) = 
    \begin{cases}
        f'(x) & \mbox{~if~} \abs{f'(x) - f'(\bar{x})} \leq \|x - \Bar{x}\|_2 \\
        f(\bar{x}) + \|x-\Bar{x}\|_2 & \mbox{~if~}f'(x) - f'(\bar{x}) > \|x - \Bar{x}\|_2. \\
        f(\bar{x}) - \|x-\Bar{x}\|_2 &\mbox{~if~} f'(x) - f'(\bar{x}) <  - \|x - \Bar{x}\|_2.
    \end{cases}
\end{align*}
Such $f$ always exists with a suitable select of $\bar{x}$, due to the $1$-Lipschitz continuous property.
For any $x \in \mX$, if $\abs{f'(x) - f'(\bar{x})} \leq \|x - \Bar{x}\|_2$ holds,
\begin{align*}
    f'(x) - f(x) = 0.
\end{align*}
If $f'(x) - f'(\bar{x}) > \|x - \Bar{x}\|_2$ holds, then
\begin{align*}
    f'(x) - f(x) &= (f'(x) - f'(\bar{x})) - ( f(x) - f(\bar{x})) \\
    &=(f'(x) - f'(\bar{x})) -\|x-\Bar{x}\|_2 \\
    & \leq \|x-\Bar{x}\|_2 + \gamma  -\|x-\Bar{x}\|_2\\
    &= \gamma.
\end{align*}
If $f'(x) - f'(\bar{x}) < - \|x - \Bar{x}\|_2$ holds, the similar result is obtained, hence we have $|f'(x) - f(x)| \leq \gamma$.
Since the result holds for any $x \in \mX$ and $f' \in \Xi^{\lip,\gamma}(L,S)$, we obtain
\begin{align}
    \sup_{f' \in \Xi^{\lip,\gamma}(L,S)} \inf_{f \in \lip(\mX)} \|f' - f\|_{L^\infty} \leq \gamma. \label{ineq:supinf}
\end{align}

Now, we bound the term \eqref{ineq:ww}.
By the H\"older's inequality, we obtain
\begin{align}
    &\sup_{f' \in \Xi^{\lip,\gamma}(L,S)} \inf_{f \in \lip(\mX)} \int (f(x) - f'(x)) d(\mu_1 - \mu_2)(x) \notag \\
    &\leq \sup_{f' \in \Xi^{\lip,\gamma}(L,S)} \inf_{f \in \lip(\mX)} \|f-f'\|_{L^\infty} \int d\abs{\mu_1 - \mu_2}(x) \notag \\
    & \leq 2\gamma, \label{ineq:ww2}
\end{align}
where the last inequality holds by \eqref{ineq:supinf} and $\int d\abs{\mu_1 - \mu_2}(x) = 2TV(\mu_1, \mu_2) \leq 2$.
Here, $TV(\cdot,\cdot)$ denotes the total variation distance.

We combine \eqref{ineq:ww1} and \eqref{ineq:ww2} with $\gamma = O(S^{-1/d})$ and set $c = 2 c'$, then we obtain the statement.
\qed

We provide an additional result for the expressive power of DNNs with the constraint on its Lipschitz continuity.
\begin{lemma} \label{lem:approx_lip}
    For any function $f^* \in \lip(\mX)$, there exists a constant $c' > 0$ such that for any $\varepsilon \in (0,1/2)$, there is a function $f \in \Xi^{\lip, \gamma}(L,S)$ with sufficiently large $S$ and $ L=O((1 + \log S) (1+\log d))$ layers and $\gamma = O(S^{-1/d})$, such as 
    \begin{align*}
        \|f^*-f\|_{L^\infty} \leq c'' S^{-1/d},
    \end{align*}
    where $c'' > 0$ is an existing constant.
\end{lemma}
\begin{proof}
This proof contains the following two steps; (i) we define an approximator $f_0: \mathbb{R}^d \to \mathbb{R}$ whose Lipschitz constant is $1$, then (ii) we show that there exists $\tilde{f} \in \Xi^{\lip,\gamma}(L,S)$ which approximates $f_0$.

Fix $f^* \in \lip(\mX)$ and $ \varepsilon \in (0,1)$ and arbitrary.
In this proof, we consider an extended version of $f^*$ from $\mX$ to $\mathbb{R}^d$, i.e., we consider $f^*: \mathbb{R}^d \to \mathbb{R}$.
Its existence is always guaranteed by Section 2.5.2 in \cite{brudnyi2011methods}.
We take $f^*$ as satisfying $\|f^*\|_{L_1(\mathbb{R}^d)} \leq c^*\|f^*\|_{L^\infty}$ with some constant $c^* > 0$.

\textbf{Step (i)}: We define $f_0 $ using a convolution for $f^*$.
For preparation, we define a kernel function $k: \mathbb{R}^d \to \mathbb{R}$ such as
\begin{align*}
    k(x) = c\exp(-1/(1-\|x\|_2^2)^2) \textbf{1}(\|x\|_2 \leq 1),
\end{align*}
where $c := (\int_{\mathbb{R}^d} k(x) dx)^{-1}$, thus $\int_{\mathbb{R}^d} k(x) = 1$ holds.
Let us define a ball $\mB_d(0,r) := \{x \in \mathbb{R}^d \mid \|x\|_2 \leq r\} \in \mathbb{R}^d$ with a radius $r > 0$.
Also, for a set $\Omega \subset \mathbb{R}^d$, $\mB_d(\Omega,r) := \{x \in \mathbb{R}^d \mid \inf_{x' \in \Omega}\|x'-x\|_2 \leq r\}$ be a neighbour of $\Omega$.
Then, we define $f_0 : \mX \to \mathbb{R}$ which is a convoluted $f^*$ with the kernel $k$ as
\begin{align*}
    f_0(x)&:= \int_{\mB_d(\mX,\varepsilon)} f^*(x') \varepsilon^{-d} k\left(\frac{ x-x'}{ \varepsilon}\right) dx' .
\end{align*}
We note that $\int_{\mathbb{R}^d} \varepsilon^{-d}k(x / \varepsilon) dx=1 $ holds.
We can simply rewrite $f_0$ as
\begin{align*}
   f_0(x)&= \varepsilon^{-d} \int_{\mB_d(0,\varepsilon)} f^*(x-x') k\left(\frac{x'}{\varepsilon}\right)dx' = \int_{\mB_d(0,1)} f^*(x-\varepsilon x')k(x')dx'.
\end{align*}
Then, we evaluate the following distance between $f_0$ and $f^*$ as
\begin{align}
    \|f_0 - f^*\|_{L^\infty} &= \sup_{x \in \mX} \abs*{ \int_{\mB_d(0,1)}(f^*(x-\varepsilon x') - f^*(x) ) k(x') dx' } \notag \\
    &\leq \sup_{x \in \mX} \sup_{x' \in \mB_d(0,1)} \abs*{ f^*(x- \varepsilon x') -  f^*(x) } \int_{\mB_d(0,1)} \abs{k(x')} dx' \notag  \\
    & \leq \varepsilon \sup_{x' \in \mB_d(0,1)} \|x\|_2  \notag \\
    &=\varepsilon. \label{ineq:approx_f0}
\end{align}
where the first inequality follows the H\"older's inequality and the second inequality follows the $1$-Lipschitz continuity of $f^*$ and $\int_{\mB_d(0,1)} \abs{k(x')} dx' \leq \int_{\mathbb{R}^d} \abs{k(x')} dx'= 1$.

\textbf{Step (ii)}: We develop $\tilde{f} \in \Xi^{\lip,\gamma}(L_1,S_1)$ for approximating $f_0$ with a tuple $(L_1,S_1)$.
Since the function $ k'(x)  := \varepsilon^{-d} k(x/\varepsilon)$ belongs to $C^\infty(\mathbb{R}^d)$, $k'$ is Lipschitz continuous and hence a slight modification of Theorem \ref{thm:NN-approx} guarantees that there exists $f_1 \in \Xi(L_1,S_1)$ such that
\begin{align}
    \| k'- f_1\|_{L^\infty(\mB_d(\mX,\varepsilon))} \leq c_1 S_1^{-1/d}, \label{ineq:kernel}
\end{align}
with an existing constant $c_1 > 0$, sufficiently large $S_1$, and $L_1 = O((1+\log S_1)(1+\log d))$.
Also, we define its convolution $f_2(x) := \int_{\mB_d(\mX,\varepsilon)} f^*(x') f_1(x-x') dx'$.

We approximate $f_0$ using $f_1$ and its convolution with a set of grid points.
Let $\mJ := \{x_j\}_{j=1}^N \subset \mB_d(\mX,\varepsilon)$ be a non-random set of $N$ equally-separated grid points.
Then, we define $\tilde{f} \in \Xi(L_1 + 1, S_1 + 2N)$ as
\begin{align*}
    \tilde{f} (x) := N^{-1} \sum_{x_j \in \mJ} w_j f_1(x - x_j),
\end{align*}
where $w_j = f^*(x_j)$ is a coefficient for $j=1,...,N$.
Since the $f^*$ is Lipschitz continuous, the classical results for numerical integration (\cite{dick2013high}) evaluates its error as
\begin{align}
     \abs*{\int_{\mB_d(\mX,\varepsilon)} f^*(x') f_1(x-x') dx'  - \tilde{f}(x)} \leq c_2 N^{-1/d}, \label{ineq:integration}
\end{align}
for any $x \in \mX$, with an existing constant $c_2 > 0$.
Combining the results \eqref{ineq:approx_f0}, \eqref{ineq:kernel} and \eqref{ineq:integration}, we obtain
\begin{align*}
    \|f^* - \tilde{f}\|_{L^\infty} &\leq \|f^* - f_0\|_{L^\infty} + \|f_0 - f_2\|_{L^\infty} + \|f_2 - \tilde{f}\|_{L^\infty}\\
    & \leq \varepsilon + \|f^*\|_{L_1(\mathbb{R}^d)} \|k' - f_1\|_{L^\infty(\mB_d(\mX,\varepsilon))} + c_1 N^{-1/d} \\
    & \leq \varepsilon + c^* \|f^*\|_{L^\infty} c_1 S_1^{-1/d} + c_2 N^{-1/d},
\end{align*}
where the second inequality follows the Young's convolution inequality.
We set $S=S_1 + 2N$ and $\varepsilon = S^{-1/d}$, then we have $\|f^* - \tilde{f}\|_{L^\infty} \leq c_3 S^{-1/d}$ with a constant $c_3 > 0$ depending on $c^* ,\|f^*\|_{L^\infty}, c_1$ and $c_2$.

Finally, we evaluate the Lipschitz continuity of $ \tilde{f}$.
For any $x_1,x_2 \in \mX$, we evaluate the following difference as follows:
\begin{align*}
    &\abs{\tilde{f}(x_1) - \tilde{f}(x_2)}\\
    &\leq \abs{f_0(x_1) - f_0(x_2)} + 2\|f_0 - \tilde{f}\|_{L^\infty} \\
    & \leq \int_{\mathbb{R}^d}\abs{f^*(x_1 - x') - f^*(x_2 -x')} k(x') dx' + 2 \|f^*\|_{L^\infty} c^* c_1 S_1^{-1/d} + 2c_2 N^{-1/d} \\
    & \leq \|x_1-x_2\|_{L^\infty} \int_{\mathbb{R}^d} k(x')dx' +  2 \|f^*\|_{L^\infty} c^* c_1 S_1^{-1/d} + 2c_2 N^{-1/d}\\
    & \leq \|x_1-x_2\|_{L^\infty}  + c_3 S^{-1/d},
\end{align*}
where the second inequality follows the $1$-Lipschitz continuity of $f^*$ and the last inequality follows $\int k(x') dx'=1$.
Then,  we obtain the statement.
\end{proof}

\subsection{Proof of Lemma \ref{lem:entropy}}

    First, by the definition of $\Xi^{\lip,\gamma}(L,S) \subset \Xi(L,S)$, we have
    \begin{align*}
        \log N(\varepsilon, \Xi^{\lip,\gamma}(L,S),\|\cdot\|) \leq \log N(\varepsilon/2, \Xi(L,S),\|\cdot\|),
    \end{align*}
    for any $L,S$ and $\varepsilon > 0$.
    Then, from Lemma 9.22 in \citep{kosorok2008introduction}, for any norm $\|\cdot\|$ dominated by $\|\cdot\|_{L^\infty}$, we have 
    \begin{align*}
        \log N (\varepsilon/2,  \Xi(L,S), \|\cdot\|) \leq \log N_{[]}(\varepsilon,  \Xi(L,S), \|\cdot\|) \leq \log N (\varepsilon/2,  \Xi(L,S), \|\cdot\|_{L^\infty}),
    \end{align*}
    where $\log N_{[]}(\varepsilon, \Omega, \|\cdot\|)$ is a bracketing nubmer of $\Omega$.
    Then, we apply Lemma \ref{lem:entropy} which bounds $\log N (\varepsilon,  \Xi(L,S), \|\cdot\|_{L^\infty})$ directly, and obtain the statement. \qed
    
    \subsection{Proof of Lemma \ref{lem:entropyint}}

    For any finite measure on $\mX$ and $\delta>0$, 
    \begin{align*}
        J(\delta) = \int_0^\delta & \sqrt{1+\log N(\varepsilon\|F_S\|_{L^2(Q)}, \Xi^{\lip,\gamma}(L,S) , \|\cdot\|_{L^2(Q)} )}d\varepsilon \\ &\leq \int_0^\delta  \sqrt{1+S\log (4(L+1)V^2/\varepsilon)}d\varepsilon.
    \end{align*}
    {Here, recall that $V$ be a product of a number of nodes in $L$ layers.
    }
    Observe that
    \begin{align*}
         \int_0^\delta  \sqrt{1+S\log (4(L+1)V^2/\varepsilon)}d\varepsilon \leq 4(L+1)V^2 \sqrt{S} \int_{4(L+1)V^2/\delta}^{\infty} \frac{\sqrt{1+\log \varepsilon}}{\varepsilon^2}d\varepsilon.
    \end{align*}
    Since $4(L+1)V^2 \geq e$ holds, an integration by parts provides the following:
    \begin{align*}
        &\int_{4(L+1)V^2/\delta}^{\infty} \frac{\sqrt{1+\log \varepsilon}}{\varepsilon^2}d\varepsilon\\
        & = \left[ 
       -\frac{\sqrt{1+\log \varepsilon}}{\varepsilon} \right]^\infty_{4(L+1)V^2/\delta} + \frac{1}{2} \int_{4(L+1)V^2/\delta}^{\infty} \frac{1}{\varepsilon^2 \sqrt{1+ \log \varepsilon}} d\varepsilon\\
       & \leq \frac{\sqrt{1 + \log (4(L+1)V^2/\delta)}}{4(L+1)V^2/\delta} + \frac{1}{2} \int_{4(L+1)V^2/\delta}^{\infty} \frac{\sqrt{1+\log \varepsilon}}{\varepsilon^2} d\varepsilon\\
       & \leq \frac{\sqrt{1 + \log (4(L+1)V^2/\delta)}}{4(L+1)V^2/\delta} + \frac{1}{2}\frac{2\sqrt{1 + \log (4(L+1)V^2/\delta)}}{4(L+1)V^2/\delta}\\
       & \leq \frac{\sqrt{2} \sqrt{\log (4(L+1)V^2)/\delta} }{4(L+1)V^2/\delta}.
    \end{align*}
    Since $4(L+1)V^2/\delta \geq 4(L+1)V^2 \geq e$, we obtain
    \[
    J(\delta) \leq C\sqrt{S} \sqrt{\delta^2 \log(1/\delta)},
    \]
    where $C$ is a sufficiently large constant.
    \qed

    \subsection{Proof of Lemma \ref{lem:intermediateGAR} }

    We consider Gaussian approximation to suprema of empirical processes indexed by the expanding function class $\Xi^{\lip,\gamma}(L, S)$, and then apply Theorem \ref{thm:CCK}.
    Since the covering number of $\Xi^{\lip,\gamma}(L,S)$ described in Lemma \ref{lem:entropy} is
    \[
    N(\varepsilon, \Xi^{\lip,\gamma}(L,S) , \|\cdot\|_{L^2(Q)}) \leq \left(\frac{4(L+1)V^2}{\varepsilon} \right)^{(S+1)},
    \]
    the function class $\Xi^{\lip,\gamma}(L,S)$ satisfies the condition in Theorem \ref{thm:CCK} with fixed parameters $(L, S)$. 
    About the moment conditions for the theorem, since the sample space $\mX$ is compact subset of $\mathbb{R}^d$ and $f \in \Xi^{\lip,\gamma}(L,S)$ is bounded function, there exist for each $S$, some $b_S \geq \sigma_S >0$, we have $\sup_{\phi \in \Xi^{\lip,\gamma}(L,S)} \Ep_\mu[\abs{f(X)}^k] \leq \sigma_S^2b_S^{k-2} $ for $k=2,3$ and $\|F_S\|_{P,q} \leq b_S$ for any $q =\infty$. Therefore, we can apply Theorem \ref{thm:CCK}.
    Here, let us define
    \begin{equation}\label{eq:K_Sn}
    K_{S,n} = c(S+1)(\log n \vee \log (4(L+1)V^2b_S/\sigma_S)),
    \end{equation}
    and will set $K_n = K_{S,n}$ as in Theorem \ref{thm:CCK}.
    
    We derive a stochastic process whose supremum can approximate $\sqrt{n/S}{ \hat{W}_1(\mu_n,\mu)}$.
    To the end, we define a Gaussian process $\{G_W(f):f \in \Xi^{\lip,\gamma}(L,S)\}$ as defined in Step 2 in Section \ref{sec:approx}.
    Since {$\hat{W}_1(\mu_n,\mu)$} is regarded as an empirical process, Theorem \ref{thm:CCK} shows
    \begin{align*}
        &\Pr\Biggl(  \abs*{ \sqrt{n} {\hat{W}_1(\mu_n,\mu)} -  \sqrt{S}\hat{Z}_W } > 
    \frac{b_S K_{S,n}}{\gamma^{1/2} n^{1/2} } + \frac{(b_S \sigma_S)^{1/2} K_{S,n}^{3/4}}{\gamma^{1/2} n^{1/4}  } + \frac{(b_S \sigma_S^2) K_{S,n}^{2/3}}{\gamma^{1/3} n^{1/6} } \Biggr)\\
    &\leq C\left( \gamma + \frac{\log n}{n}  \right),
    \end{align*}
    for every $\gamma \in (0,1)$.
    The multiplier $\sqrt{S}$ for $\hat{Z}_{W}$ comes from a scaled covariance function of the Gaussian process $G_W(f)$.
    By dividing the both hand sides inside the probability by $\sqrt{S}$, we obtain the following coupling result
    \begin{align*}
        &\Pr\left( \abs*{ \sqrt{\frac{n}{S}}{\hat{W}_1(\mu_n,\mu)} -  \hat{Z}_W } > 
        \frac{b_S K_{S,n}}{\gamma^{1/2} n^{1/2} \sqrt{S}} + \frac{(b_S \sigma_S)^{1/2} K_{S,n}^{3/4}}{\gamma^{1/2} n^{1/4}  \sqrt{S}} + \frac{(b_S \sigma_S^2) K_{S,n}^{2/3}}{\gamma^{1/3} n^{1/6} \sqrt{S}} \right)\\
        &\leq C\left( \gamma + \frac{\log n}{n}  \right),
    \end{align*}
    where $c, C > 0$ are positive constants.
    Here, we know that $b_S$ is finite due to the boundedness of $f \in \Xi^{\lip,\gamma}(L,S)$ and $K_{S,n} = O(S \log n)$.
    As taking $\gamma = O((\log n)^{-1})$, we obtain
    \begin{align*}
        &\Pr\left( \abs*{ \sqrt{\frac{n}{S}}{\hat{W}_1(\mu_n,\mu)} -  \hat{Z}_W } > 
        \frac{  c_W'' S^{1/2} \log^{1/2} n}{ n^{1/2} } + \frac{c_W'' \sigma_S^{1/2} S^{1/4}  \log^{5/4} n}{ n^{1/4}  } + \frac{c_W'' \sigma_S^2 S^{1/6} \log n}{ n^{1/6} } \right)\\
        &\leq C'\left( \frac{1}{\log n} + \frac{\log n}{n}  \right),
    \end{align*}
    where $C',c_W'' > 0$ are constants depending on $b_S$.
    Regardless the constants, the term $\frac{S^{1/6} \log n}{n^{1/6}}$ is larger than $\frac{S^{1/2} \log^{1/2}n}{n^{1/2}}$ and $\frac{S^{1/4} \log^{5/4}n}{n^{1/4}}$ as $S/n \to 0$, hence we obtain the statement with $c_W = 3c_W''$.
    \qed

    \subsection{Proof of Theorem \ref{thm:GAR_kolmogorov_2}}
    We combine the results by Lemma \ref{lem:NN-approx} and Lemma \ref{lem:intermediateGAR}.
    The target probability is bounded as
    \begin{align*}
        &\Pr\left( \abs*{ \sqrt{\frac{n}{S}}{W_1(\mu_n, \mu)} - \hat{Z}_W } > \frac{c_W \sigma^2S^{1/6} \log n}{n^{1/6}}  \right)\\& \leq \Pr\left( \abs*{ \sqrt{\frac{n}{S}}{\hat{W}_1(\mu_n, \mu)}  - \hat{Z}_W } > \frac{c_W \sigma^2S^{1/6}\log n}{2n^{1/6}}  \right)\\
        &\quad + \Pr\left(\sqrt{\frac{n}{S}} \abs*{ {\hat{W}_1(\mu_n, \mu)}  -  {W_1(\mu_n, \mu)} } > \frac{c_W \sigma^2S^{1/6}\log n}{2n^{1/6}}\right) \\
        & \leq C'' \left( \frac{1}{\log n} + \frac{\log n}{n}  \right)+  \frac{2c_W\Ep_\mu[\abs{{\hat{W}_1(\mu_n, \mu)} - {W_1(\mu_n, \mu)}}]}{c_W \sigma^2} \left( \frac{n}{S}\right)^{2/3} \log^2n,
    \end{align*}
    where $C'' > 0$ is an existing finite constant depends on $n$.
    About the last inequality, the first term comes from Lemma \ref{lem:intermediateGAR}, and the second term follows the Markov's inequality. Applying Lemma \ref{lem:NN-approx}, we get
    \[
    {\Ep[\abs{\hat{W}_1(\mu_n, \mu) - W_1(\mu_n, \mu)}] \leq c S^{-1/d}}. 
    \]
    Finally, we obtain
\begin{align*}
    &\Pr\left( \abs*{\sqrt{\frac{n}{S}}{W_1(\mu_n, \mu)} - \hat{Z}_W } > \frac{c_W \sigma^2S^{1/6}}{n^{1/6}}  \right) \\
    &\leq C'' \left( \frac{1}{\log n} + \frac{\log n}{n}\right) + 2c_W c\left(\frac{n \log^{1/3}n}{S^{(2d+3)/2d}} \right)^{2/3},
\end{align*}
    then the right hand side converges to $0$ as $n \to \infty$
    under the condition on $S$. 
    \qed
    
    \subsection{Proof of Corollary \ref{cor:main} }
    
    We bound the Kolmogorov distance between $\sqrt{n/S}{W_1(\mu_n,\mu)}$ and $\hat{Z}_W$.
    To the end, we will show $\Pr ( \hat{Z}_W \leq z ) \leq \Pr ( \sqrt{n/S} {W_1(\mu_n,\mu)} \leq z ) + o(1)$.
    For any $z \in \mathbb{R}$ and some $z'$, we have
    \begin{align*}
        &\Pr \left( \hat{Z}_W \leq z \right) \\
        & \leq \Pr \left(\left\{ \hat{Z}_W \leq z \right\} \bigcap \left\{ \abs*{ \sqrt{\frac{n}{S}}{W_1(\mu_n,\mu)} - \hat{Z}_W  } \leq z'\right\} \right)\\
        & \quad + \Pr \left(\left\{ \hat{Z}_W \leq z \right\} \bigcap \left\{ \abs*{ \sqrt{\frac{n}{S}}{W_1(\mu_n,\mu)} - \hat{Z}_W  } > z'\right\} \right)\\
        &\leq \Pr \left( \sqrt{\frac{n}{S}}{W_1(\mu_n,\mu)} \leq z + z' \right) + \Pr \left( \abs*{ \sqrt{\frac{n}{S}}{W_1(\mu_n,\mu)} - \hat{Z}_W  } > z' \right).
    \end{align*}
    About the second probability, we have
    \begin{align*}
        &\Pr \left( \abs*{ \sqrt{\frac{n}{S}}{W_1(\mu_n,\mu)} - \hat{Z}_W  } > z' \right) \\
        & \leq \Pr \left(  \sqrt{\frac{n}{S}}\abs*{ {W_1(\mu_n,\mu) -  \hat{W}_1(\mu_n,\mu)} } > \frac{1}{2} z' \right) + \Pr \left( \abs*{ \sqrt{\frac{n}{S}}{\hat{W}_1(\mu_n,\mu)} - \hat{Z}_W  } >  \frac{1}{2} z' \right).
    \end{align*}
    Here, we set $z' = 2 S^{1/6} \log^2n/n^{1/6}$, then the two terms are $o(1)$ with the settings on $S$.
    About the first probability $\Pr ( \sqrt{n/S}{W_1(\mu_n,\mu)} \leq z + z' )$, by Assumption \ref{asmp:cont_cdf}, we obtain $\abs{\Pr ( \sqrt{n/S}{W_1(\mu_n,\mu)} \leq z + z' ) - \Pr ( \sqrt{n/S}{W_1(\mu_n,\mu)} \leq z  ) } = o(1)$ as $z' = o(1)$.
    
    An opposite inequality $\Pr ( \hat{Z}_W \leq z ) \geq \Pr ( \sqrt{n/S} {W_1(\mu_n,\mu)} \leq z ) + o(1)$ is obtained by a similar way.
    \qed

    \subsection{Proof of Theorem \ref{thm:GMB}}
    
    From Theorem 2.2 in \cite{chernozhukov2016empirical}, under the condition $K_{S,n} \leq n$.
    For every $\gamma \in (0,1)$ and $\eta>0$, 
    there exists a random variable $\tilde{Z}_W$ with fixed $\mD_n$ such that 
    \[
    P\left(\abs*{\hat{Z}_W - \tilde{Z}_W  } > c_3 (\eta + \delta^{(2)}_{S,n})/S  \right) \leq c_4 (\gamma + n^{-1}),
    \]
    where $c_3$ and $c_4$ are universal constants, and $ \delta^{(2)}_{S,n}$ is defined as follows:
    \[
    \delta^{(2)}_{S,n}=\frac{b_S K_{S,n}}{\gamma n^{1/2}} + \frac{(b_S\sigma K_{S,n}^{3/2})^{1/2}}{\gamma n^{1/4}}.
    \]
    The rest of this proof follows a similar way in the proof of Theorem \ref{thm:GAR_kolmogorov_2} and Corollary \ref{cor:main}.
    \qed

    \subsection{Proof of Theorem \ref{thm:one_test}}
    
    \begin{proof}
     We will show that the one-sample test has asymptotically size $\alpha$, i.e.,
    \begin{align*}
    \Pr\left(\sqrt{\frac{n}{S_n}} {W_1(\mu_n,\mu_0)} > \tilde{q}(1-\alpha) \right) =  \alpha + o(1), ~~(n \to \infty).
    \end{align*}
    
    Under the null hypothesis $H_0: \mu = \mu_0$, 
    we approximate the asymptotic distribution of the test statistic {$W_1(\mu_n, \mu_0)$} by the distribution of the corresponding multiplier bootstrap process, i.e., $\tilde{Z}_W$.
    From Theorem \ref{thm:GMB}, we have 
    \begin{align*}
    \Pr\left(\sqrt{\frac{n}{S_n}} {W_1(\mu_n,\mu_0)} > \tilde{q}(1-\alpha) \right) & = \Pr\left(\tilde{Z}_W > \tilde{q}(1-\alpha) \mid \mD_n^X\right) + o(1)\\
    & = \alpha + o(1),
    \end{align*}
    since $\tilde{q}(1-\alpha)$ is the $1-\alpha$ quantile of the bootstrap distribution $\tilde{Z}_W$.
    \end{proof}

    \subsection{Proof of Theorem \ref{thm:two_test}}

    \begin{proof}
    In the following, we consider $\lambda_{n,m} \to \lambda \in (0,1)$ as $n,m \to \infty$ and fix $r \in (0,1)$ arbitrary.
     We will show that the two-sample test has asymptotically size $\alpha$, i.e.,
    \begin{equation}\label{eq:null_one}
        \Pr \left(\sqrt{\frac{\rho_{n,m}}{S_n \wedge S_m}} {W_1(\mu_n,\nu_n)} > \breve{q}(1-\alpha) \right) \leq \alpha +o(1), ~~(n,m \to \infty).
    \end{equation}
     Under the null hypothesis $H_0: \mu = \nu$, 
    we will approximate the asymptotic distribution of the test statistic $W(\mu_n, \nu_m)$ by the distribution of the sum of two multiplier bootstrap processes, i.e., $\sqrt{\lambda_{n,m}}\tilde{Z}_{X,W} + \sqrt{1-\lambda_{n,m}}\tilde{Z}_{Y,W}$. 
    To this end, we use the triangle inequality
    {
    \begin{align*}
    W_1(\mu_n,\nu_n) &\leq  W_1(\mu_n,\mu) +  W_1(\nu_n,\nu) +  W_1(\mu,\nu)\\
    & = W_1(\mu_n,\mu) +  W_1(\nu_n,\nu).
    \end{align*}
    }
    Then, we have
    \begin{align*}
    &\Pr \left(\sqrt{\frac{\rho_{n,m}}{S_n \wedge S_m}} {W_1(\mu_n,\nu_n)} > \breve{q}(1-\alpha) \right)\\
    & \leq \Pr \left(\sqrt{\frac{\rho_{n,m}}{S_n \wedge S_m}} {W_1(\mu_n,\mu)} +  \sqrt{\frac{\rho_{n,m}}{S_n \wedge S_m}} W(\nu_n,\nu) > \breve{q}(1-\alpha) \right).\\
    \end{align*} 
    Since $\rho_{n,m} = (nm)/(m+n)$ and $\lambda_{n,m} := m/(n+m)$, {
    \begin{align*}
       \sqrt{\frac{\rho_{n,m}}{S_n \wedge S_m}} W_1(\mu_n,\mu) + \sqrt{\frac{\rho_{n,m}}{S_n \wedge S_m}}W_1(\nu_n,\nu) =  \sqrt{\lambda_{n,m}} W_1(\mu_n,\mu) + \sqrt{1-\lambda_{n,m}} W_1(\nu_m, \nu).
    \end{align*}}
    Then, we apply Theorem \ref{thm:GMB} to $ \sqrt{\frac{n}{S_n}} {W_1(\mu_n,\mu)}$ and $\sqrt{\frac{m }{S_n}} {W_1(\nu_m, \nu)}$ and obtain
    \begin{align*}
        &\Pr \left(\sqrt{\frac{\rho_{n,m}}{S_n \wedge S_m}} {W_1(\mu_n,\mu)} +  \sqrt{\frac{\rho_{n,m}}{S_n \wedge S_m}} {W_1(\nu_n,\nu)} > \breve{q}(1-\alpha) \right) \\& \leq   \Pr\left( \sqrt{\lambda_{n,m}} \tilde{Z}_{X, W}> \breve{q}_X(1-r\alpha)  \mid \mD_n^X \right) \\
        & \quad + \Pr\left( \sqrt{1-\lambda_{n,m}} \tilde{Z}_{Y, W}>   \breve{q}_Y(1-(1-r)\alpha)   \mid  \mD_m^Y\right)+ o(1)\\
        & \leq  \alpha + o(1).
    \end{align*}
    \end{proof}

\section{Additional Experiments}

\subsection{Numerical analysis for CDF of Wasserstein distances} \label{sec:asmp1}

We numerically validate the anti-concentration assumption of {$W_1(\mu_n,\mu)$} (Assumption \ref{asmp:cont_cdf}). To this end, we consider six settings described in Figure \ref{image:converge} and \ref{image:converge2}. In both of the figures, the top left panel shows a case when $\mu$ is the standard normal distribution. The top center panel shows a case when $\mu$ is the Laplace distribution with location parameter 0 and scale parameter 1. The top right panel shows a case when $\mu$ is a Gaussian mixture. The bottom left shows a case when $\mu$ is multinomial distribution with parameters $(p_1, p_2, p_3, p_4) = (1/20, 3/20, 6/20, 10/20)$. The bottom center panel shows a case when $\mu$ is the Poisson distribution with parameter $\lambda = 4$. The bottom right panel shows a case when $\mu$ is the 3-dimensional multivariate normal distribution with mean zero and identity covariance matrix. The Wasserstein distance is computed using the \texttt{R} package \texttt{transport}. The number of repetitions is 2000, and the sample size is 1000 for each case.

We plot the probability density function/histogram of {$W_1(\mu_n, \mu)$} in Figure \ref{image:converge}. The numerical experiments show that the probability density function of {$W_1(\mu_n, \mu)$} is likely to be bounded. Moreover, we compute the constant factor of Assumption \ref{image:converge2}, $C_r := \Pr(r \leq {W_1(\mu_n, \mu)} \leq r + \delta)/\delta$ for {$\delta \in \{2^{-1}, 2^{-2}, 2^{-3}, 2^{-4}, 2^{-5}, 2^{-6}, 2^{-7}\}$.}, respectively. The Figure \ref{image:converge2} indicates that the constant factor $C_r$ does not tend to be large. These two results sufficiently support our assumption numerically.

\begin{figure}[htbp]
    \centering
    \begin{minipage}{1\hsize}
        \centering
        \includegraphics[width=.450\linewidth]{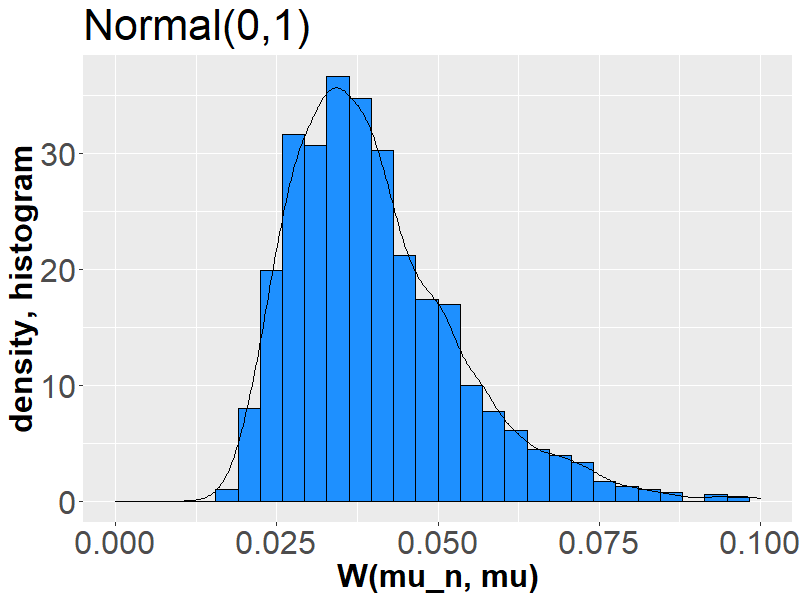}
        \includegraphics[width=.450\linewidth]{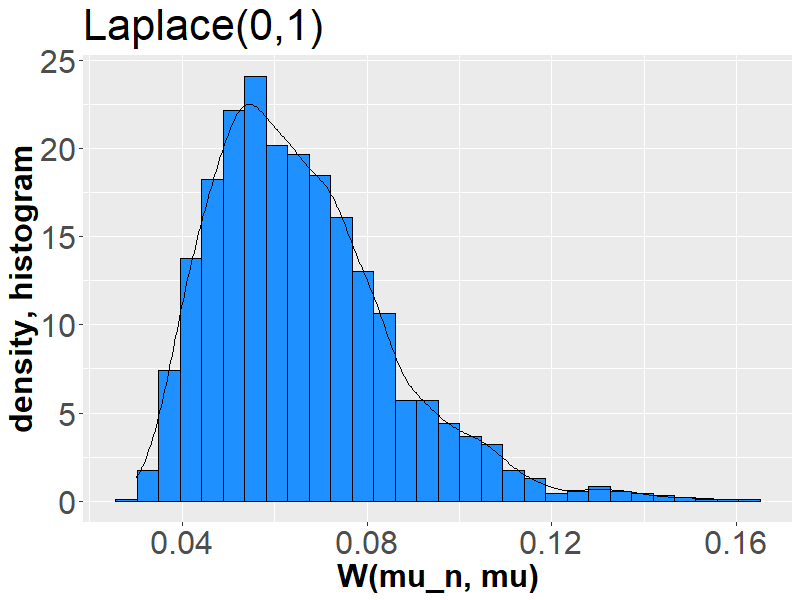}\\
        \includegraphics[width=.450\linewidth]{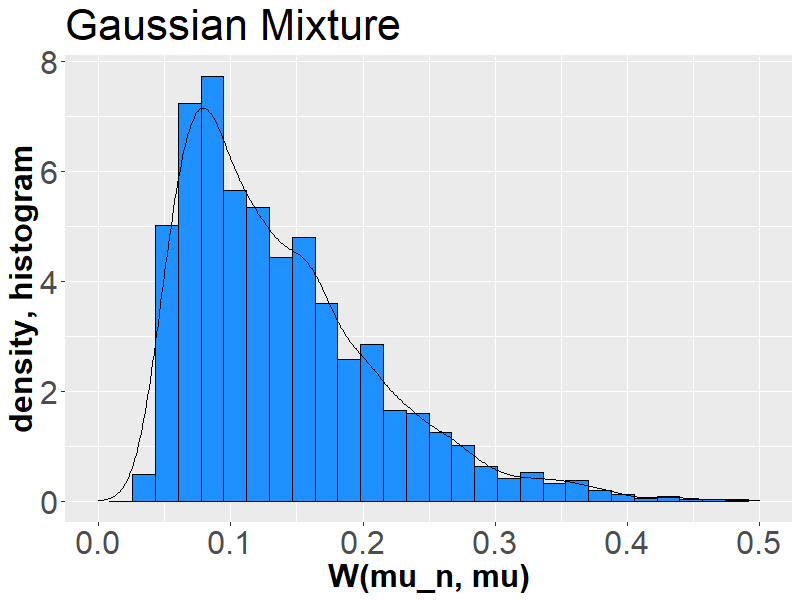}
        \includegraphics[width=.450\linewidth]{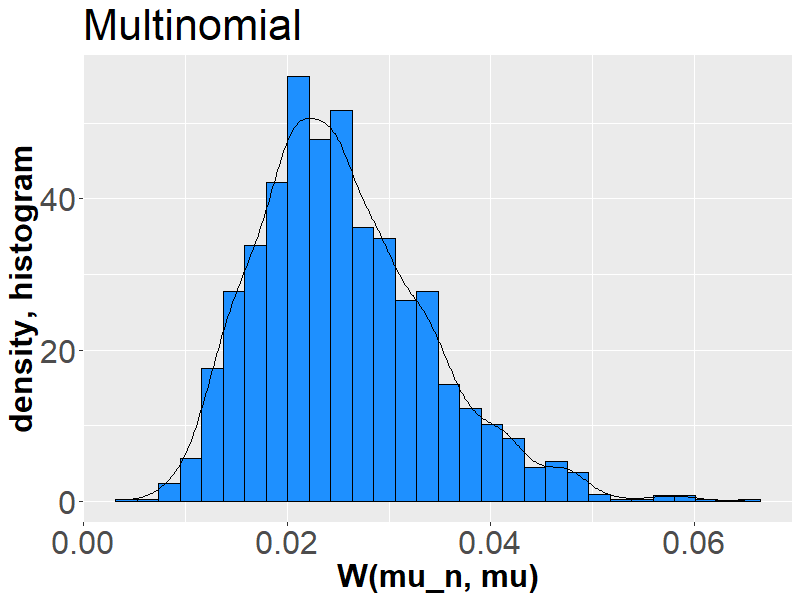}\\
        \includegraphics[width=.450\linewidth]{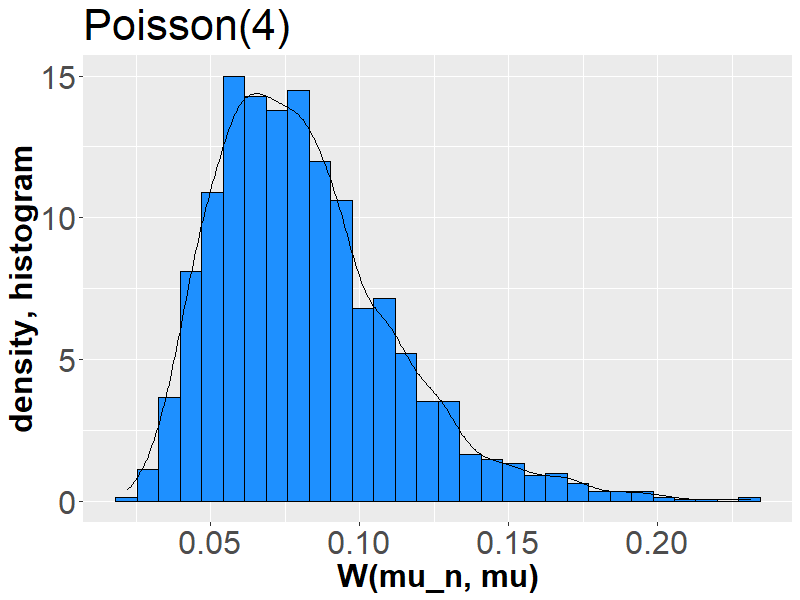}
        \includegraphics[width=.450\linewidth]{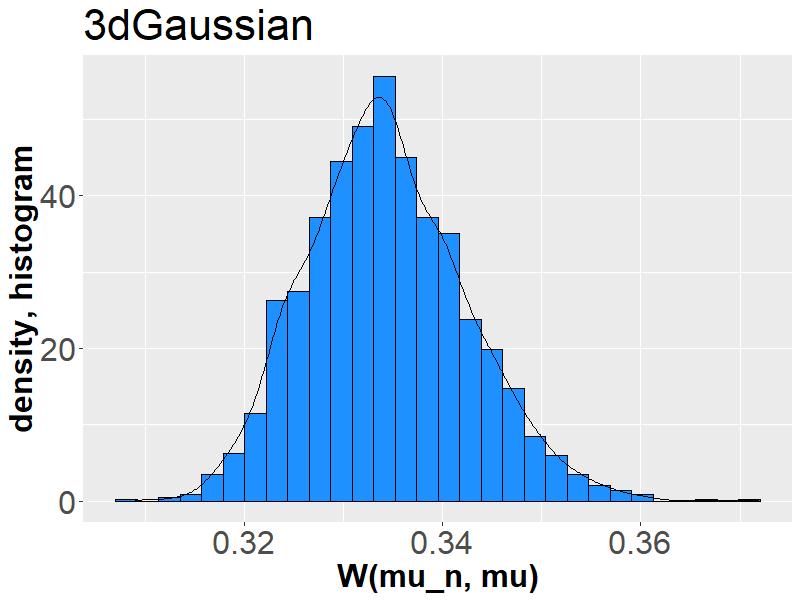}
        \caption{Density/histogram plots of {$W_1(\mu_n, \mu)$}.\label{image:converge}}
    \end{minipage}
\end{figure}

\begin{figure}[htbp]
    \centering
    \begin{minipage}{1\hsize}
        \centering
        \includegraphics[width=.450\linewidth]{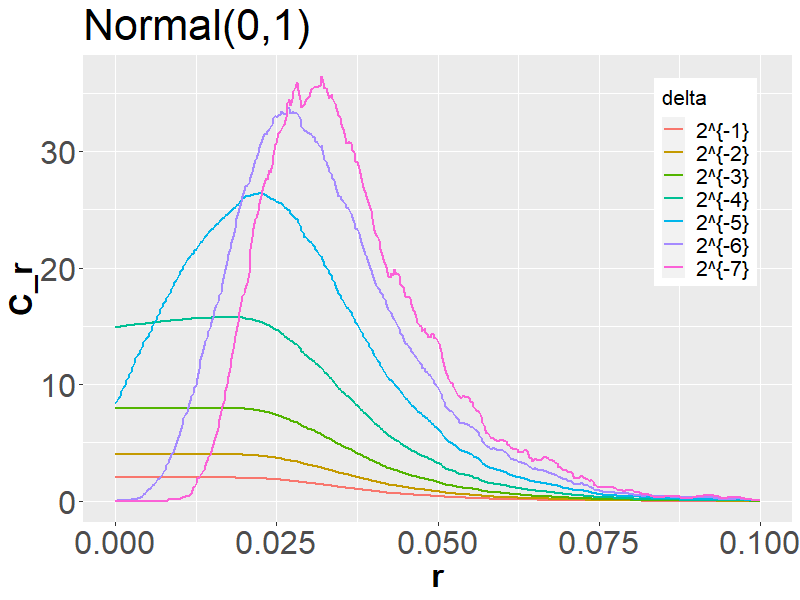}
        \includegraphics[width=.450\linewidth]{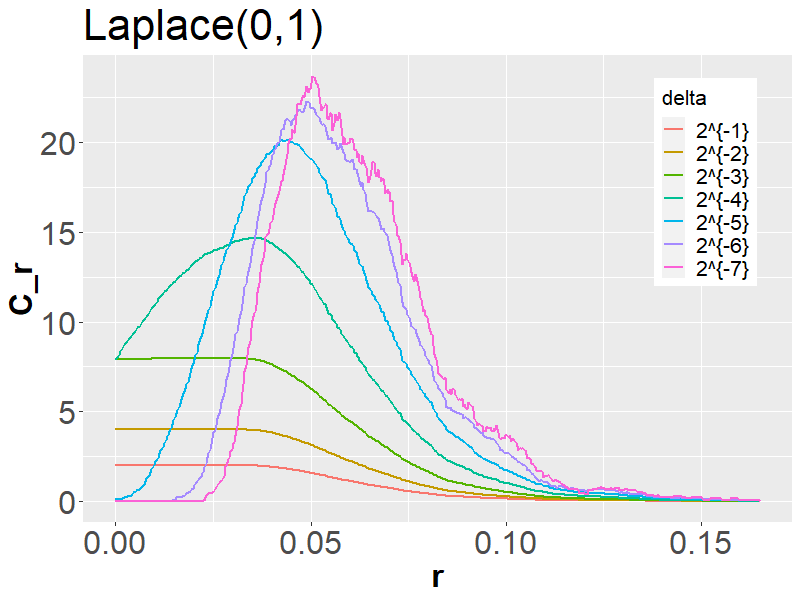}\\
        \includegraphics[width=.450\linewidth]{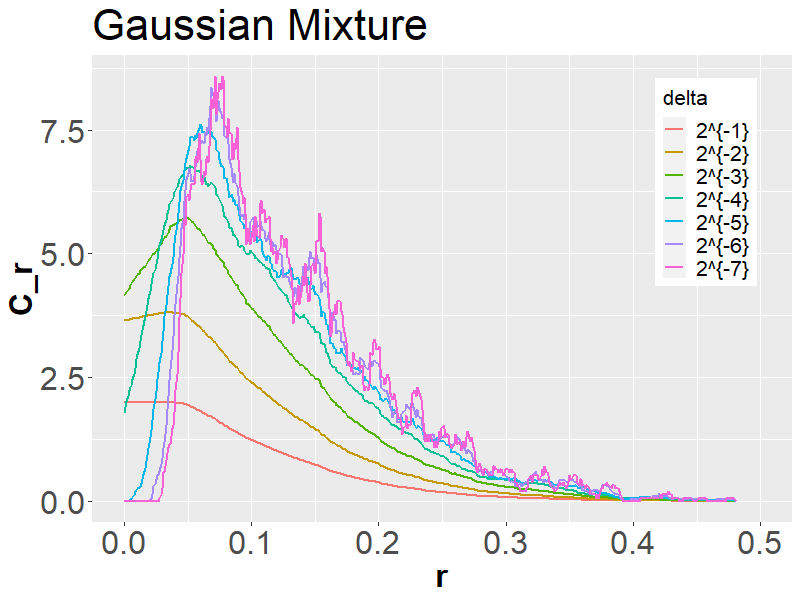}
        \includegraphics[width=.450\linewidth]{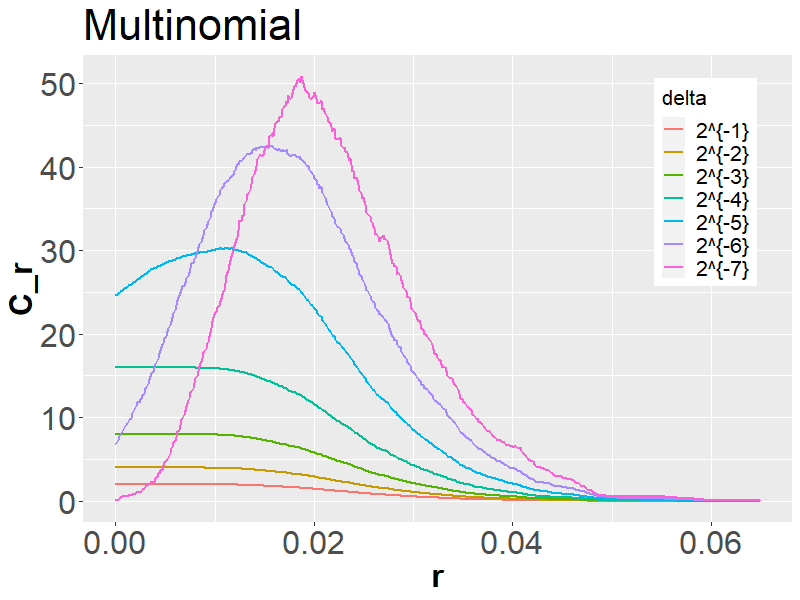}\\
        \includegraphics[width=.450\linewidth]{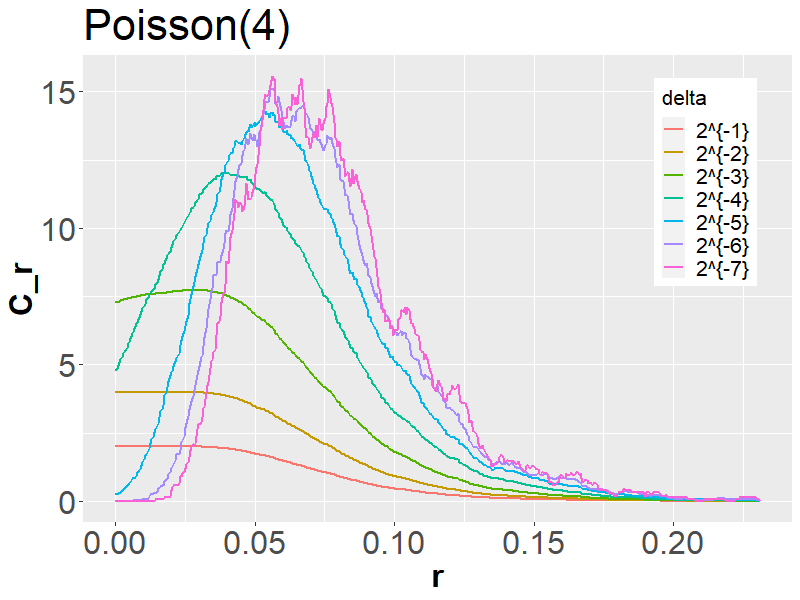}
        \includegraphics[width=.450\linewidth]{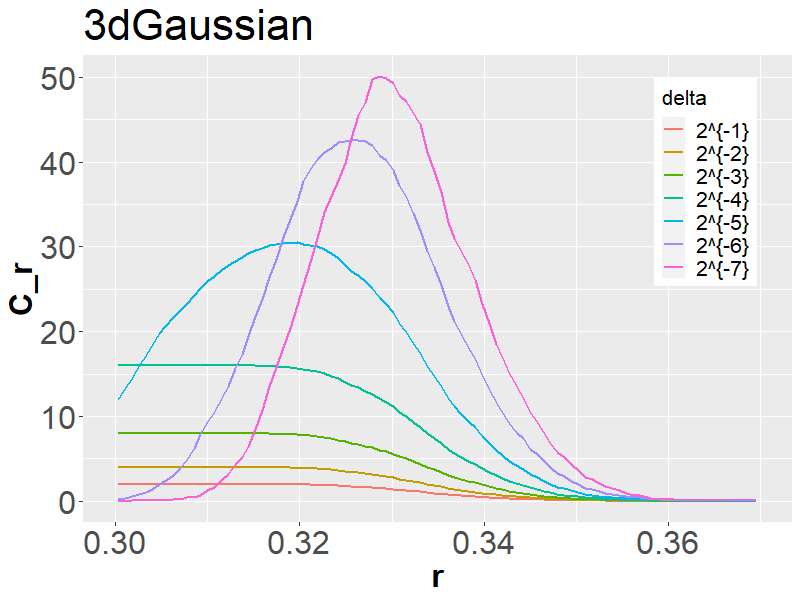}
        \caption{{Plots of $ C_r := \Pr(r \leq W(\mu,\mu_n) \leq r+\delta)$ against $r$, for each $\delta \in \{2^{-1}, 2^{-2}, 2^{-3}, 2^{-4}, 2^{-5}, 2^{-6}, 2^{-7}\}$.}} \label{image:converge2}
    \end{minipage}
\end{figure}

{
\subsection{Comparison in Univariate Case}

We analyze whether the proposed method is consistent with existing studies in the univariate setting. 
As described in Section \ref{sec:related}, the asymptotic distribution of the empirical Wasserstein distance for $d = 1$ is known analytically. 
We verify whether our proposed method is consistent in this setting by simulation.

We numerically reproduce the distribution of the empirical Wasserstein distance by a baseline from existing studies and by the proposed method.
For the true distribution $\mu$, as in Section \ref{sec:exp_validation}, we use three types of distributions: a normal, an exponential, and a mixture of Gaussians. 
For the generation of the distribution by the proposed method, we adopt exactly the same setup as in Section \ref{sec:exp_validation}. 
For the baseline approach, we compute the distribution numerically by repeating the data generation $2000$ times based on the asymptotic distribution in \cite{del2019central}.

Fig \ref{fig:univariate} shows density functions of the derived distribution.
The densities are computed by the kernel density estimation as in Section \ref{sec:exp_validation}. 
In all the cases, the baseline and the proposed method have distributions that are close enough.

\begin{figure}
    \centering    \includegraphics[width=0.32\hsize]{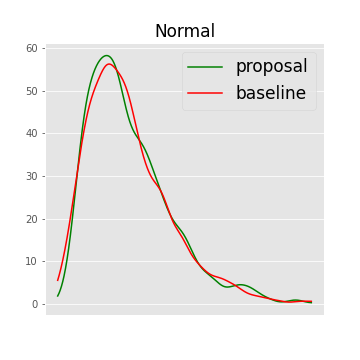}
    \includegraphics[width=0.32\hsize]{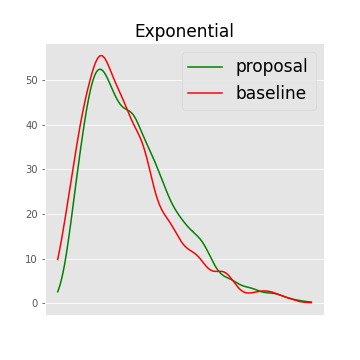}
    \includegraphics[width=0.32\hsize]{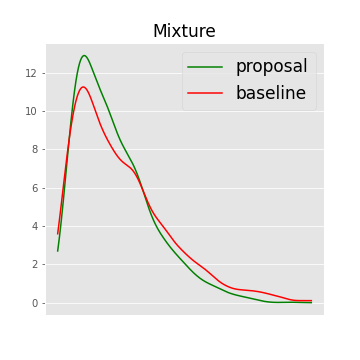}
    \caption{{Density of distributions of the empirical Wasserstein distance with $d=1$. The green curve is our proposal, and the red curve is the baseline derived from \citet{del2019central}.}}
    \label{fig:univariate}
\end{figure}

\subsection{Effect from Neural Network Architecture} \label{sec:select_architecture}

{
We verify the robustness of our proposed method under different neural network architectures by testing whether the proposed bootstrap procedure can obtain a coherent distribution. 
In the theoretical result (Lemma \ref{lem:intermediateGAR}, the existence of a suitable neural network is guaranteed, however, it is difficult to choose the proper architecture directly. 
Therefore, we conduct a comprehensive experiment to investigate the impact of network architecture.

The architectural choices we study are a combination of the following conditions: the number of samples $n \in \{10,100,1000\}$, the number of hidden layers $L \in  \{1,2,...,6\}$, and the dimension of intermediate variable (width) in each layer $\mathrm{dim} \in \{1,2,5,10,50,100\}$.
These combinations bring the number of architecture types to $108$.
The minimum number of parameters is $3$ and the maximum is $50300$.
In addition, although outside the scope of the theoretical result, we study the residual network (ResNet) architecture \cite{he2016deep}, and also investigated the case where the activation function is a hyperbolic tangent ($\tanh$) for both ordinary neural networks and residual networks.
For all other settings, we employ the same setting in Section \ref{sec:exp_validation}.

Figure \ref{fig:different_arch_dnn} shows density functions of the obtained distribution for approximation with neural networks with ReLU or $\tanh$ activation functions.
Figure \ref{fig:different_arch_dnn} shows the same results when using ResNet.
$n$ is the number of samples, \textit{Layers} is a number of hidden layers $L$, \textit{dim} is the dimension (width) of each layer, and \textit{Param} is the total number of parameters in neural networks.
The densities are computed by the kernel density estimation as in Section \ref{sec:exp_validation}.
We consider three different distributions of data: Gaussian, exponential, and mixture, but since the results are similar, only the results for the Gaussian case are presented.

From these results, we obtain the following implications:
(i) As long as \textit{dim} is above $50$, the proposed method is stable regardless of $n$ and \textit{Layer}.
(ii) The increase in $n$ contributes to the stabilization.
(iii) The increase of the number of hidden layers causes the instability when $n$ and \textit{dim} are small.
(iv) \textit{Param} is important, but fewer parameters still produce stable results when \textit{dim} is large.
These results are common to ordinary neural networks and residual networks, and are not affected by the choice of activation functions.
From these results, we recommend increasing the dimension of each layer and limiting the number of hidden layers. 
Also, a wide variety of network types are allowed to be selected.

}

\begin{figure}[htbp]
    \centering
    \includegraphics[width=\hsize]{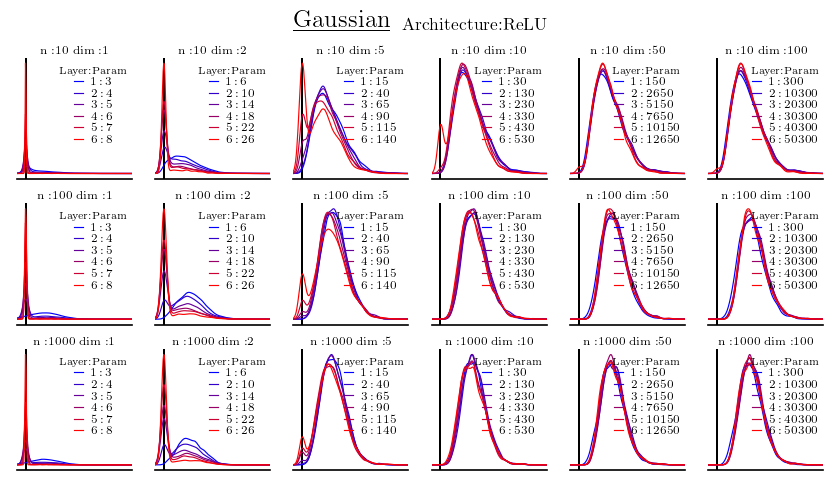}\\
    \includegraphics[width=\hsize]{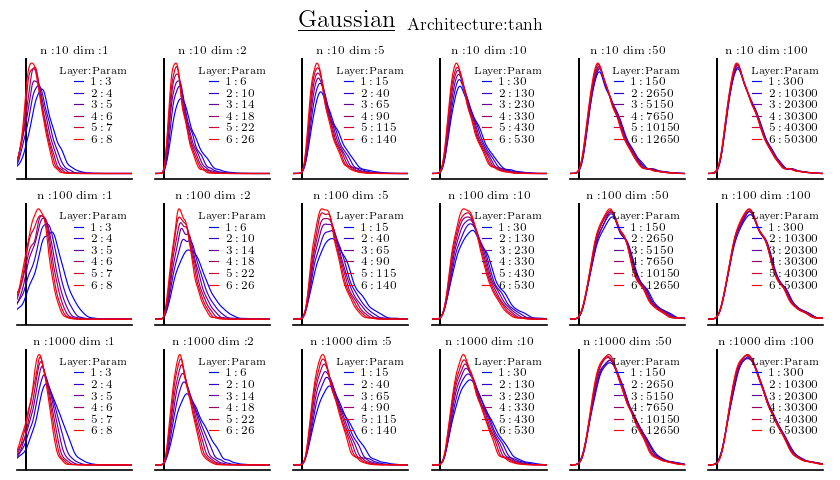}\\
    \caption{Density of distributions by our proposed method with neural networks. The samples are Gaussian, and activation functions are ReLU or $\tanh$. $n$ is the number of samples, \textit{Layers} is a number of hidden layers $L$, \textit{dim} is the dimension (width) of each layer, and \textit{Param} is the total number of parameters.}
    \label{fig:different_arch_dnn}
\end{figure}

\begin{figure}[htbp]
    \centering
    \includegraphics[width=\hsize]{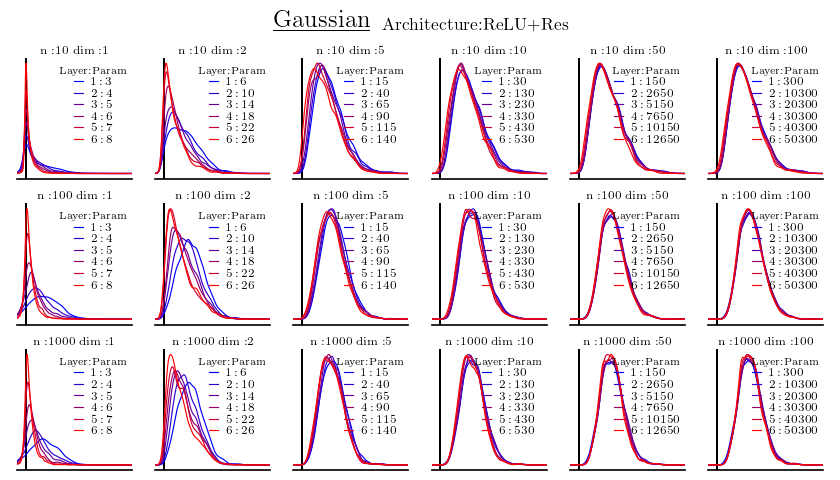}\\
    \includegraphics[width=\hsize]{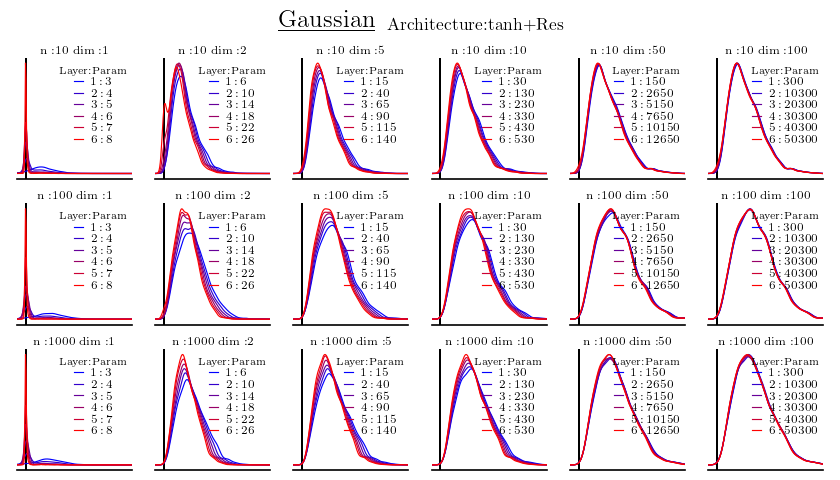}\\
    \caption{Density of distributions by our proposed method with residual networks (\textit{Res}). The samples are Gaussian, and activation functions are ReLU or $\tanh$. $n$ is the number of samples, \textit{Layers} is a number of hidden layers $L$, \textit{dim} is the dimension (width) of each layer, and \textit{Param} is the total number of parameters.}
    \label{fig:different_arch_res}
\end{figure}

}

\bibliographystyle{apa}
\bibliography{wasserstein}

\begin{thebibliography}{}

\bibitem[\protect\astroncite{Ambrosio et~al.}{2008}]{ambrosio2008gradient}
Ambrosio, L., Gigli, N., and Savar{\'e}, G. (2008).
\newblock {\em Gradient flows: in metric spaces and in the space of probability
  measures}.
\newblock Springer Science \& Business Media.

\bibitem[\protect\astroncite{Arjovsky et~al.}{2017}]{arjovsky2017wasserstein}
Arjovsky, M., Chintala, S., and Bottou, L. (2017).
\newblock Wasserstein generative adversarial networks.
\newblock In {\em International Conference on Machine Learning}, pages
  214--223. PMLR.

\bibitem[\protect\astroncite{Bentkus}{2005}]{bentkus2005lyapunov}
Bentkus, V. (2005).
\newblock {A Lyapunov-type bound in Rd}.
\newblock {\em Theory of Probability \& Its Applications}, 49(2):311--323.

\bibitem[\protect\astroncite{Bernton et~al.}{2017}]{bernton2017inference}
Bernton, E., Jacob, P.~E., Gerber, M., and Robert, C.~P. (2017).
\newblock {Inference in generative models using the Wasserstein distance}.
\newblock {\em arXiv preprint arXiv:1701.05146}.

\bibitem[\protect\astroncite{Bigot et~al.}{2017}]{bigot2017central}
Bigot, J., Cazelles, E., and Papadakis, N. (2017).
\newblock {Central limit theorems for Sinkhorn divergence between probability
  distributions on finite spaces and statistical applications}.
\newblock {\em arXiv preprint arXiv:1711.08947}.

\bibitem[\protect\astroncite{Bigot et~al.}{2019}]{bigot2019central}
Bigot, J., Cazelles, E., and Papadakis, N. (2019).
\newblock Central limit theorems for entropy-regularized optimal transport on
  finite spaces and statistical applications.
\newblock {\em Electronic Journal of Statistics}, 13(2):5120--5150.

\bibitem[\protect\astroncite{Brudnyi}{2011}]{brudnyi2011methods}
Brudnyi, A. (2011).
\newblock {\em Methods of Geometric Analysis in Extension and Trace Problems:
  Volume 1}, volume 102.
\newblock Springer Science \& Business Media.

\bibitem[\protect\astroncite{Chen and Friedman}{2017}]{chen2017new}
Chen, H. and Friedman, J.~H. (2017).
\newblock A new graph-based two-sample test for multivariate and object data.
\newblock {\em Journal of the American statistical association},
  112(517):397--409.

\bibitem[\protect\astroncite{Chen and Niles-Weed}{2021}]{chen2021asymptotics}
Chen, H.-B. and Niles-Weed, J. (2021).
\newblock Asymptotics of smoothed {W}asserstein distances.
\newblock {\em Potential Analysis}, pages 1--25.

\bibitem[\protect\astroncite{Chernozhukov
  et~al.}{2013}]{chernozhukov2013gaussian}
Chernozhukov, V., Chetverikov, D., and Kato, K. (2013).
\newblock Gaussian approximations and multiplier bootstrap for maxima of sums
  of high-dimensional random vectors.
\newblock {\em The Annals of Statistics}, 41(6):2786--2819.

\bibitem[\protect\astroncite{Chernozhukov
  et~al.}{2014}]{chernozhukov2014gaussian}
Chernozhukov, V., Chetverikov, D., and Kato, K. (2014).
\newblock Gaussian approximation of suprema of empirical processes.
\newblock {\em The Annals of Statistics}, 42(4):1564--1597.

\bibitem[\protect\astroncite{Chernozhukov
  et~al.}{2015}]{chernozhukov2015comparison}
Chernozhukov, V., Chetverikov, D., and Kato, K. (2015).
\newblock Comparison and anti-concentration bounds for maxima of gaussian
  random vectors.
\newblock {\em Probability Theory and Related Fields}, 162(1-2):47--70.

\bibitem[\protect\astroncite{Chernozhukov
  et~al.}{2016}]{chernozhukov2016empirical}
Chernozhukov, V., Chetverikov, D., and Kato, K. (2016).
\newblock Empirical and multiplier bootstraps for suprema of empirical
  processes of increasing complexity, and related gaussian couplings.
\newblock {\em Stochastic Processes and their Applications},
  126(12):3632--3651.

\bibitem[\protect\astroncite{Cohen et~al.}{1993}]{cohen1993wavelets}
Cohen, A., Daubechies, I., and Vial, P. (1993).
\newblock Wavelets on the interval and fast wavelet transforms.
\newblock {\em Applied and computational harmonic analysis}.

\bibitem[\protect\astroncite{Del~Barrio et~al.}{2000}]{del2000contributions}
Del~Barrio, E., Cuesta-Albertos, J.~A., Matr{\'a}n, C., Cs{\"o}rg{\"o}, S.,
  Cuadras, C.~M., de~Wet, T., Gin{\'e}, E., Lockhart, R., Munk, A., and Stute,
  W. (2000).
\newblock Contributions of empirical and quantile processes to the asymptotic
  theory of goodness-of-fit tests.
\newblock {\em Test}, 9(1):1--96.

\bibitem[\protect\astroncite{del Barrio et~al.}{1999}]{del1999tests}
del Barrio, E., Cuesta-Albertos, J.~A., Matr{\'a}n, C., and
  Rodr{\'\i}guez-Rodr{\'\i}guez, J.~M. (1999).
\newblock {Tests of goodness of fit based on the L2-Wasserstein distance}.
\newblock {\em Annals of Statistics}, pages 1230--1239.

\bibitem[\protect\astroncite{Del~Barrio et~al.}{2019}]{del2019central}
Del~Barrio, E., Gordaliza, P., Lescornel, H., and Loubes, J.-M. (2019).
\newblock {Central limit theorem and bootstrap procedure for Wasserstein’s
  variations with an application to structural relationships between
  distributions}.
\newblock {\em Journal of Multivariate Analysis}, 169:341--362.

\bibitem[\protect\astroncite{Del~Barrio and Loubes}{2019}]{del2019central2}
Del~Barrio, E. and Loubes, J.-M. (2019).
\newblock Central limit theorems for empirical transportation cost in general
  dimension.
\newblock {\em The Annals of Probability}, 47(2):926--951.

\bibitem[\protect\astroncite{Dick et~al.}{2013}]{dick2013high}
Dick, J., Kuo, F.~Y., and Sloan, I.~H. (2013).
\newblock {High-dimensional integration: the quasi-Monte Carlo way}.
\newblock {\em Acta Numerica}, 22:133.

\bibitem[\protect\astroncite{Dudley}{2002}]{dudley_2002}
Dudley, R.~M. (2002).
\newblock {\em Real Analysis and Probability}.
\newblock Cambridge Studies in Advanced Mathematics. Cambridge University
  Press, 2 edition.

\bibitem[\protect\astroncite{Evans and Matsen}{2012}]{evans2012phylogenetic}
Evans, S.~N. and Matsen, F.~A. (2012).
\newblock {The phylogenetic Kantorovich--Rubinstein metric for environmental
  sequence samples}.
\newblock {\em Journal of the Royal Statistical Society: Series B (Statistical
  Methodology)}, 74(3):569--592.

\bibitem[\protect\astroncite{Fournier and Guillin}{2015}]{fournier2015rate}
Fournier, N. and Guillin, A. (2015).
\newblock {On the rate of convergence in Wasserstein distance of the empirical
  measure}.
\newblock {\em Probability Theory and Related Fields}, 162(3-4):707--738.

\bibitem[\protect\astroncite{Frogner et~al.}{2015}]{frogner2015learning}
Frogner, C., Zhang, C., Mobahi, H., Araya, M., and Poggio, T.~A. (2015).
\newblock {Learning with a Wasserstein loss}.
\newblock In {\em Advances in Neural Information Processing Systems}, pages
  2053--2061.

\bibitem[\protect\astroncite{Goldfeld et~al.}{2020}]{goldfeld2020asymptotic}
Goldfeld, Z., Greenewald, K., and Kato, K. (2020).
\newblock Asymptotic guarantees for generative modeling based on the smooth
  {W}asserstein distance.
\newblock {\em Advances in neural information processing systems}.

\bibitem[\protect\astroncite{Gretton et~al.}{2012}]{gretton2012kernel}
Gretton, A., Borgwardt, K.~M., Rasch, M.~J., Sch{\"o}lkopf, B., and Smola, A.
  (2012).
\newblock A kernel two-sample test.
\newblock {\em Journal of Machine Learning Research}, 13(Mar):723--773.

\bibitem[\protect\astroncite{Hallin et~al.}{2021}]{hallin2021multivariate}
Hallin, M., Mordant, G., and Segers, J. (2021).
\newblock Multivariate goodness-of-fit tests based on {W}asserstein distance.
\newblock {\em Electronic Journal of Statistics}, 15(1):1328--1371.

\bibitem[\protect\astroncite{He et~al.}{2016}]{he2016deep}
He, K., Zhang, X., Ren, S., and Sun, J. (2016).
\newblock Deep residual learning for image recognition.
\newblock In {\em Proceedings of the IEEE conference on computer vision and
  pattern recognition}, pages 770--778.

\bibitem[\protect\astroncite{Imaizumi and Fukumizu}{2019}]{imaizumi2019deep}
Imaizumi, M. and Fukumizu, K. (2019).
\newblock Deep neural networks learn non-smooth functions effectively.
\newblock In {\em The 22nd international conference on artificial intelligence
  and statistics}, pages 869--878. PMLR.

\bibitem[\protect\astroncite{Imaizumi and
  Fukumizu}{2020}]{imaizumi2020advantage}
Imaizumi, M. and Fukumizu, K. (2020).
\newblock Advantage of deep neural networks for estimating functions with
  singularity on curves.
\newblock {\em arXiv preprint arXiv:2011.02256}.

\bibitem[\protect\astroncite{Kim et~al.}{2020}]{kim2020robust}
Kim, I., Balakrishnan, S., and Wasserman, L. (2020).
\newblock Robust multivariate nonparametric tests via projection averaging.
\newblock {\em The Annals of Statistics}, 48(6):3417--3441.

\bibitem[\protect\astroncite{Kingma and Ba}{2015}]{adam_opt}
Kingma, D.~P. and Ba, J. (2015).
\newblock Adam: {A} method for stochastic optimization.
\newblock In {\em International Conference on Learning Representations 2015s}.

\bibitem[\protect\astroncite{Kosorok}{2008}]{kosorok2008introduction}
Kosorok, M.~R. (2008).
\newblock {\em Introduction to empirical processes and semiparametric
  inference.}
\newblock Springer.

\bibitem[\protect\astroncite{LeCun et~al.}{1998}]{lecun1998gradient}
LeCun, Y., Bottou, L., Bengio, Y., and Haffner, P. (1998).
\newblock Gradient-based learning applied to document recognition.
\newblock {\em Proceedings of the IEEE}, 86(11):2278--2324.

\bibitem[\protect\astroncite{Liberzon et~al.}{2011}]{liberzon2011molecular}
Liberzon, A., Subramanian, A., Pinchback, R., Thorvaldsd{\'o}ttir, H., Tamayo,
  P., and Mesirov, J.~P. (2011).
\newblock {Molecular signatures database (MSigDB) 3.0}.
\newblock {\em Bioinformatics}, 27(12):1739--1740.

\bibitem[\protect\astroncite{Lin et~al.}{2021}]{lin2021projection}
Lin, T., Zheng, Z., Chen, E., Cuturi, M., and Jordan, M. (2021).
\newblock On projection robust optimal transport: {S}ample complexity and model
  misspecification.
\newblock In {\em International Conference on Artificial Intelligence and
  Statistics}, pages 262--270. PMLR.

\bibitem[\protect\astroncite{Lloyd and Ghahramani}{2015}]{lloyd2015statistical}
Lloyd, J.~R. and Ghahramani, Z. (2015).
\newblock Statistical model criticism using kernel two sample tests.
\newblock In {\em Advances in Neural Information Processing Systems}, pages
  829--837.

\bibitem[\protect\astroncite{Massey~Jr}{1951}]{massey1951kolmogorov}
Massey~Jr, F.~J. (1951).
\newblock {The Kolmogorov-Smirnov test for goodness of fit}.
\newblock {\em Journal of the American statistical Association},
  46(253):68--78.

\bibitem[\protect\astroncite{Mena and Weed}{2019}]{mena2019statistical}
Mena, G. and Weed, J. (2019).
\newblock Statistical bounds for entropic optimal transport: sample complexity
  and the central limit theorem.
\newblock {\em Advances in Neural Information Processing Systems}.

\bibitem[\protect\astroncite{Miyato et~al.}{2018}]{miyato2018spectral}
Miyato, T., Kataoka, T., Koyama, M., and Yoshida, Y. (2018).
\newblock Spectral normalization for generative adversarial networks.
\newblock In {\em International Conference on Learning Representations}.

\bibitem[\protect\astroncite{Munk and Czado}{1998}]{munk1998nonparametric}
Munk, A. and Czado, C. (1998).
\newblock Nonparametric validation of similar distributions and assessment of
  goodness of fit.
\newblock {\em Journal of the Royal Statistical Society: Series B (Statistical
  Methodology)}, 60(1):223--241.

\bibitem[\protect\astroncite{Nadjahi et~al.}{2019}]{nadjahi2019asymptotic}
Nadjahi, K., Durmus, A., Simsekli, U., and Badeau, R. (2019).
\newblock {Asymptotic Guarantees for Learning Generative Models with the
  Sliced-Wasserstein Distance}.
\newblock {\em Advances in Neural Information Processing Systems}, 32:250--260.

\bibitem[\protect\astroncite{Ni et~al.}{2009}]{ni2009local}
Ni, K., Bresson, X., Chan, T., and Esedoglu, S. (2009).
\newblock {Local histogram based segmentation using the Wasserstein distance}.
\newblock {\em International journal of computer vision}, 84(1):97--111.

\bibitem[\protect\astroncite{Panaretos and
  Zemel}{2018}]{panaretos2018statistical}
Panaretos, V.~M. and Zemel, Y. (2018).
\newblock {Statistical aspects of Wasserstein distances}.
\newblock {\em Annual Review of Statistics and Its Application}.

\bibitem[\protect\astroncite{Ramdas et~al.}{2017}]{ramdas2017wasserstein}
Ramdas, A., Trillos, N., and Cuturi, M. (2017).
\newblock {On Wasserstein two-sample testing and related families of
  nonparametric tests}.
\newblock {\em Entropy}, 19(2):47.

\bibitem[\protect\astroncite{Rosenbaum}{2005}]{rosenbaum2005exact}
Rosenbaum, P.~R. (2005).
\newblock An exact distribution-free test comparing two multivariate
  distributions based on adjacency.
\newblock {\em Journal of the Royal Statistical Society: Series B (Statistical
  Methodology)}, 67(4):515--530.

\bibitem[\protect\astroncite{Ruttenberg
  et~al.}{2013}]{ruttenberg2013quantifying}
Ruttenberg, B.~E., Luna, G., Lewis, G.~P., Fisher, S.~K., and Singh, A.~K.
  (2013).
\newblock Quantifying spatial relationships from whole retinal images.
\newblock {\em Bioinformatics}, 29(7):940--946.

\bibitem[\protect\astroncite{Schmidt-Hieber}{2020}]{schmidt2017nonparametric}
Schmidt-Hieber, J. (2020).
\newblock {Nonparametric regression using deep neural networks with ReLU
  activation function}.
\newblock {\em The Annals of Statistics}, 48(4):1875--1897.

\bibitem[\protect\astroncite{Sommerfeld and
  Munk}{2018}]{sommerfeld2018inference}
Sommerfeld, M. and Munk, A. (2018).
\newblock {Inference for empirical Wasserstein distances on finite spaces}.
\newblock {\em Journal of the Royal Statistical Society: Series B (Statistical
  Methodology)}, 80(1):219--238.

\bibitem[\protect\astroncite{Song}{2002}]{song2002goodness}
Song, K.-S. (2002).
\newblock {Goodness-of-fit tests based on Kullback-Leibler discrimination
  information}.
\newblock {\em IEEE Transactions on Information Theory}, 48(5):1103--1117.

\bibitem[\protect\astroncite{Tameling et~al.}{2019}]{tameling2017empirical}
Tameling, C., Sommerfeld, M., and Munk, A. (2019).
\newblock {Empirical optimal transport on countable metric spaces:
  Distributional limits and statistical applications}.
\newblock {\em The Annals of Applied Probability}, 29(5):2744--2781.

\bibitem[\protect\astroncite{van~der Vaart and Wellner}{1996}]{vdVW1996}
van~der Vaart, A. and Wellner, J. (1996).
\newblock {\em Weak Convergence and Empirical Processes: With Applications to
  Statistics}.
\newblock Springer.

\bibitem[\protect\astroncite{Vaserstein}{1969}]{vaserstein1969markov}
Vaserstein, L.~N. (1969).
\newblock Markov processes over denumerable products of spaces, describing
  large systems of automata.
\newblock {\em Problemy Peredachi Informatsii}, 5(3):64--72.

\bibitem[\protect\astroncite{Vasicek}{1976}]{vasicek1976test}
Vasicek, O. (1976).
\newblock A test for normality based on sample entropy.
\newblock {\em Journal of the Royal Statistical Society: Series B
  (Methodological)}, 38(1):54--59.

\bibitem[\protect\astroncite{Villani}{2008}]{villani2008optimal}
Villani, C. (2008).
\newblock {\em Optimal transport: old and new}, volume 338.
\newblock Springer Science \& Business Media.

\bibitem[\protect\astroncite{Weed and Bach}{2019}]{weed2019sharp}
Weed, J. and Bach, F. (2019).
\newblock {Sharp asymptotic and finite-sample rates of convergence of empirical
  measures in Wasserstein distance}.
\newblock {\em Bernoulli}, 25(4A):2620--2648.

\end{thebibliography}
\end{document}